\documentclass[11pt,reqno]{amsart}
\usepackage[a4paper,margin=29mm,headheight=14pt]{geometry}
\usepackage[T1]{fontenc}
\usepackage{lmodern}
\usepackage{amsmath,amssymb,amsthm,mathtools,microtype}
\usepackage{xcolor}
\usepackage[colorlinks=true,linkcolor=black,citecolor=black,urlcolor=blue!55!black]{hyperref}
\numberwithin{equation}{section}
\newtheorem{theorem}{Theorem}[section]
\newtheorem{lemma}[theorem]{Lemma}
\newtheorem{proposition}[theorem]{Proposition}

\theoremstyle{definition}
\newtheorem{definition}[theorem]{Definition}
\theoremstyle{remark}
\newtheorem{remark}[theorem]{Remark}
\DeclareMathOperator{\tr}{Tr}
\allowdisplaybreaks[1]
\title[Single commutators in Schatten ideals]{Single commutators in Schatten ideals}
\author{Zhichao Liu}
\address{School of Mathematical Sciences,
Dalian University of Technology,
Dalian, {\rm 116024}, China }
\email{zhichaoliu@dlut.edu.cn}
\author{Shanshan Yang}

\address{School of Mathematical Sciences,
Dalian University of Technology,
Dalian, {\rm 116024}, China }
\email{yss@mail.dlut.edu.cn}
\date{}
\hypersetup{pdftitle={Single commutators in Schatten ideals},
  pdfauthor={Zhichao Liu and Shanshan Yang},pdfsubject={Version 3; revised introduction and background}}

\begin{document}
\begin{abstract}
Let $r>2$ and let $p,q\in(0,\infty)$ satisfy
$r^{-1}=p^{-1}+q^{-1}$.
We prove that every $T\in \mathcal{S}_r(\mathcal H)$ is a single commutator
$T=[A,B]$ with $A\in \mathcal{S}_p(\mathcal H)$ and $B\in \mathcal{S}_q(\mathcal H)$.
Combined with L. G. Brown's theorem for $1 \leq r \leq 2$, this completely resolves the Pearcy--Topping problem for Schatten ideals.
We also obtain simultaneous operator and Schatten norm bounds
for the factors.
\end{abstract}
\maketitle
\enlargethispage{9pt}

\section{Introduction}\label{sec:intro}
\subsection{Background}
Let $\mathcal H$ be a complex separable infinite-dimensional Hilbert
space. We write $\mathcal B(\mathcal H)$ and $\mathcal K(\mathcal H)$
for the bounded and compact operators on $\mathcal H$, respectively. We write
$[A,B]=AB-BA$ for the commutator of $A$ and $B$.
For $1\le r<\infty$, let $\mathcal{S}_r(\mathcal H)$ denote the Schatten
$r$-class, with norm $\|\cdot\|_r$.

The commutator problem takes different forms in finite and infinite
dimensions. A finite matrix is a commutator precisely when its trace
is zero. On an infinite-dimensional Hilbert space, Wintner
\cite{Wintner} and Wielandt \cite{Wielandt} proved that the identity
is not a commutator of bounded operators. Halmos \cite{Halmos1963}
applied this result in the quotient by the compact operators to
show that no operator of the form $\lambda I+K$, with
$\lambda\ne0$ and $K$ compact, is a commutator.
Brown, Halmos and Pearcy \cite{BHP} proved
that every compact operator is a commutator of bounded operators.
In 1965, Brown and Pearcy \cite{BrownPearcy} completed the
classification by showing that the operators $\lambda I+K$ above
are the only exceptions. Their proof combines geometric reductions
under similarity with block matrix equations.

In 1971, Pearcy and Topping \cite{PT} raised the question of whether one can construct such single commutators strictly within specific ideals. They also proved that every compact operator is a finite sum of commutators of compact operators and asked whether one commutator always suffices. For Schatten classes, they obtained a partial result: if $r>1$, every $T\in \mathcal{S}_r(\mathcal H)$ is a finite sum of commutators with both factors in $\mathcal{S}_{2r}(\mathcal H)$. They also asked whether every trace-zero operator in
$\mathcal{S}_1(\mathcal H)$ is a commutator of Hilbert--Schmidt operators.

A major step toward solving the single-commutator problem was achieved by Anderson \cite{Anderson} in 1977, when he represented a rank-one orthogonal projection as a commutator of compact operators. Furthermore, his block construction applies to any compact operator with an infinite-dimensional reducing null subspace. Indeed, his theorem serves as the cornerstone of the present paper.

In 1980, Weiss \cite{Weiss} demonstrated that the trace-zero condition is insufficient to characterize even a finite sum of commutators of Hilbert--Schmidt operators. Consequently, the finite-dimensional trace criterion does not extend to this choice of factor ideals. That same year, Fan and Fong \cite{FF} characterized operators of the form $[A^*,A]$ for compact $A$, addressing the more restrictive case where the two factors are mutually adjoint.

In the case of Schatten ideals, L. G. Brown \cite{Brown} took a significant step forward in 1994 by identifying a further restriction on the Schatten exponents. He proved that if $A\in \mathcal{S}_p(\mathcal H)$, $B\in \mathcal{S}_q(\mathcal H)$,
$p^{-1}+q^{-1}\ge1/2$, and $[A,B]$ has finite rank, then
$\tr([A,B])=0$. 
In particular, Anderson's rank-one example cannot have both factors in $\mathcal{S}_4$.
More generally, when $1\le r\le2$, a rank-one orthogonal projection
cannot be written as $[A,B]$ with $A,B\in \mathcal{S}_{2r}(\mathcal H)$.
The remaining question is whether every $T\in \mathcal{S}_r(\mathcal H)$ admits such a
representation when $r>2$.

\vskip 2mm
{\bf Question (Pearcy--Topping).} For $r>2$, does every
$T\in \mathcal{S}_r(\mathcal H)$ admit $A,B\in \mathcal{S}_{2r}(\mathcal H)$ such
that $T=[A,B]$?
\vskip 2mm


More generally, the study of the single-commutator question for compact ideals also spans several decades. Belti\c{t}\u{a}, Patnaik and Weiss \cite{BPW} modified Anderson's construction to obtain
strictly positive compact examples. They also recorded earlier
unpublished examples of Davidson, Marcoux and Radjavi \cite{DMR}. In 2004, Dykema, Figiel, Weiss, and Wodzicki \cite{DFWW} developed a general commutator-space theory through the study of finite sums.
Much later, Dykema and Krishnaswamy-Usha \cite{DKU} showed that every compact nilpotent operator is a commutator of compact operators.  Recently, Loreaux,
Patnaik, Petrovi\'c and Weiss \cite{LPPW} reformulated the problem through finite block equations and studied how block sizes constrain
ideal membership of the factors.
More recently, Liu \cite{Liu} settled the Pearcy--Topping problem for the compact ideal by showing that every compact operator is a commutator of compact operators. The compression and trace-adjustment arguments in that work offer valuable insights for our proof. By integrating these techniques with the weighted construction described above, we resolve the Pearcy--Topping problem for Schatten ideals.

\subsection{Main result}

In this work, we obtain a strengthened version of this result for two prescribed factor ideals. The relation $\mathcal{S}_p(\mathcal H)\mathcal{S}_q(\mathcal H)=\mathcal{S}_r(\mathcal H)$ for $r^{-1}=p^{-1}+q^{-1}$ suggests
the appropriate exponents. Our theorem realizes every element of
$\mathcal{S}_r$ as a single commutator with factors in $\mathcal{S}_p$ and $\mathcal{S}_q$,
and gives simultaneous bounds in the operator and Schatten norms.

\begin{theorem}\label{thm:main}
Let $r>2$, and let $p,q\in(0,\infty)$ satisfy
$r^{-1}=p^{-1}+q^{-1}$.
For every $T\in \mathcal{S}_r(\mathcal H)$, there exist
$A\in \mathcal{S}_p(\mathcal H)$ and $B\in \mathcal{S}_q(\mathcal H)$ such that
$T=[A,B]$ and, simultaneously,
\[
 \|A\|\|B\|\le6\|T\|,
 \qquad
 \|A\|_p\|B\|_q
 \le16\zeta(r-1)^{1/r}\|T\|_r.
\]
Here $\zeta(r-1)=\sum_{n=1}^{\infty}n^{1-r}$.
\end{theorem}


The case $p=q=2r$ answers the Pearcy-Topping question above. Brown's result shows that the restriction $r>2$ is necessary. More precisely, for
$1\le r<\infty$ and finite positive $p,q$ satisfying
$r^{-1}=p^{-1}+q^{-1}$, one has
\[
 \bigl\{[A,B]:A\in \mathcal{S}_p(\mathcal H),\ B\in \mathcal{S}_q(\mathcal H)\bigr\}
 =\mathcal{S}_r(\mathcal H)
 \quad\Longleftrightarrow\quad r>2.
\]
H\"older's inequality gives the inclusion into $\mathcal{S}_r$,
Theorem~\ref{thm:main} gives the reverse inclusion for $r>2$,
and a rank-one orthogonal projection excludes equality for
$1\le r\le2$.

The dependence on $r$ is necessary for the Schatten norm estimate; see Proposition \ref{prop:lower}. For a rank-one orthogonal projection $P$, every representation $P=[A,B]$ with $A\in \mathcal{S}_p(\mathcal H)$ and $B\in \mathcal{S}_q(\mathcal H)$  satisfies
\[
 \|A\|_{p}\|B\|_{q}\ge\frac{1}{2}\zeta(r-1)^{1/r}.
\]

\subsection{The method}

Assume that $p\le q$. For an operator with an infinite-dimensional reducing null subspace, the argument is essentially completed within the framework of Anderson's construction (see Theorem \ref{lem:stable}). 
The remaining argument proceeds by applying this method to one diagonal corner, whereas the other is dealt with by finite-dimensional approximation.

For an infinite-rank operator $T$, utilizing the compression and trace adjustment arguments from Section \ref{sec:geometry}, Lemma \ref{lem:preparation} establishes
\[
 S^{-1}TS=\begin{pmatrix}\widehat D&E\\F&G\end{pmatrix}
 \begin{matrix}\mathcal H_0\\\mathcal H_1\end{matrix},
 \qquad
 G=\begin{pmatrix}
 \lambda&0&0&\cdots\\
 0&G_1&0&\cdots\\
 0&0&G_2&\ddots\\
 \vdots&\vdots&\ddots&\ddots
 \end{pmatrix}.
\]
Here $\widehat D$ has an infinite-dimensional reducing null subspace, $G\in \mathcal{S}_1$ and $\tr G=0$. Applying Theorem~\ref{lem:stable} to the first
corner gives
\[
 \widehat D=[C,Z],\qquad C\in \mathcal{S}_p(\mathcal H_0),
 \quad Z\in \mathcal{S}_q(\mathcal H_0).
\]

In constructing the representation for $G$, a sequence $(\lambda_n) \in \ell_p$ belonging to the resolvent set $\rho(C)$ is selected according to Lemma \ref{lem:circle}. An initial regular pair is provided by Lemma \ref{lem:seed}. The key novel tool used in the proof is the concept of regular pairs.
Let $\widehat G_n$ be the trace-adjusted compressions given in Lemma~\ref{lem:trace-truncations}. At every stage $n$, we have
\[
 \left(
 \begin{pmatrix}A_{n-1}&0\\0&\lambda_n\end{pmatrix},
 \begin{pmatrix}B_{n-1}&0\\0&0\end{pmatrix}
 \right)
 \longmapsto(A_n,B_n),
 \qquad [A_n,B_n]=\widehat G_n.
\]
The successive corrections are summable in \(\mathcal S_p(\mathcal H)\), and
Lemma~\ref{lem:approximation} yields
\[
 A_n\xrightarrow{\mathcal{S}_p}A_G,\qquad
 B_n\xrightarrow{\mathcal{S}_p}B_G,\qquad [A_G,B_G]=G.
\]

It remains to determine the off-diagonal blocks. The following operator equations can be solved via a standard calculation (Lemma \ref{lem:coupling}):
\[
 CX-XA_G=E,\qquad A_GY-YC=F,
 \qquad X,Y\in \mathcal{S}_1.
\]
Hence
\[
 \left[
 \begin{pmatrix}C&0\\0&A_G\end{pmatrix},
 \begin{pmatrix}Z&X\\Y&B_G\end{pmatrix}
 \right]
 =\begin{pmatrix}\widehat D&E\\F&G\end{pmatrix}
 =S^{-1}TS.
\]
Conjugating by $S$ gives the required representation of $T$.

\section{Preliminaries}\label{sec:prelim}

In this paper, all Hilbert spaces are complex separable, and inner products are linear in the first variable.
\subsection{Schatten ideals}
\begin{definition}
For Hilbert spaces $\mathcal H_1,\mathcal H_2$, let $\mathcal{B}(\mathcal H_2,\mathcal H_1)$ and $\mathcal K(\mathcal H_2,\mathcal H_1)$ be the bounded and compact operators from $\mathcal H_2$ to $\mathcal H_1$, respectively. If $\mathcal H=\mathcal H_1=\mathcal H_2$, we write $\mathcal B(\mathcal H)$ and $\mathcal K(\mathcal H)$. For $T\in \mathcal B(\mathcal H_2,\mathcal H_1)$, write $T^*$ for its adjoint and $\|T\|$ for its operator norm. For $T\in \mathcal B(\mathcal H)$, write $\sigma(T)$ for its spectrum.
\end{definition}
\begin{definition}
For $T\in \mathcal B(\mathcal H)$, set $|T|=(T^*T)^{1/2}$. The operator $T$ is {\em trace class} if
\begin{equation*}
\sum_{j\in J}\langle |T|e_j,e_j\rangle<\infty
\end{equation*}
for an orthonormal basis $\{e_j\}_{j\in J}$. The trace class is denoted by $\mathcal{S}_1(\mathcal H)$. For $T\in \mathcal{S}_1(\mathcal H)$, its trace and trace norm are
\begin{equation*}
\operatorname{Tr}T=\sum_{j\in J}\langle Te_j,e_j\rangle,\qquad \|T\|_1=\operatorname{Tr}|T|.
\end{equation*}
The trace series converges absolutely and is independent of the basis. For an operator $T$ on a finite-dimensional space, we write $\operatorname{Tr} T$ for the unnormalized trace.
\end{definition}
We recall the trace formulation of the Schatten classes and their norms; see \cite{McCarthy,Simon}.
\begin{definition}
Let $\mathcal H_1,\mathcal H_2$ be Hilbert spaces and $1\le p<\infty$. The {\em Schatten $p$-class} $\mathcal{S}_p(\mathcal H_2,\mathcal H_1)$ is defined by
\begin{equation*}
\mathcal{S}_p(\mathcal H_2,\mathcal H_1)=\{T\in \mathcal K(\mathcal H_2,\mathcal H_1):|T|^p\in \mathcal{S}_1(\mathcal H_2)\},
\end{equation*}
where $|T|=(T^*T)^{1/2}$. This class is equipped with the Schatten $p$-norm
\begin{equation*}
\|T\|_p=\bigl(\operatorname{Tr}(|T|^p)\bigr)^{1/p},
\end{equation*}
where the trace is taken on $\mathcal H_2$. When $\mathcal H_1=\mathcal H_2=\mathcal H$, we write $\mathcal{S}_p(\mathcal H)=\mathcal{S}_p(\mathcal H,\mathcal H)$.
\end{definition}

All Schatten norms are unnormalized. We record some useful properties.

\begin{proposition}\label{prop:norms}
{\rm (1)} If $T\in \mathcal{S}_1(\mathcal H)$, then
$$
|\operatorname{Tr}(T)|\le\|T\|_1.
$$
{\rm (2)} If $1\le p<\infty$, $T\in \mathcal{S}_p(\mathcal H)$, and
$U,V\in \mathcal B(\mathcal H)$, then $UTV\in \mathcal{S}_p(\mathcal H)$ and
\begin{equation*}
\|UTV\|_p\le\|U\|\|T\|_p\|V\|,\qquad \|T\|\leq\|T\|_p.
\end{equation*}
{\rm (3)} Suppose $1\le p,q,r<\infty$ and
$r^{-1}=p^{-1}+q^{-1}$.
For $A\in \mathcal{S}_p(\mathcal H)$ and $B\in \mathcal{S}_q(\mathcal H)$,
the operators $AB$ and $BA$ belong to $\mathcal{S}_r(\mathcal H)$, with
\begin{equation*}
\|AB\|_r\le\|A\|_p\|B\|_q.
\end{equation*}
{\rm (4)} If $1\le p\le q<\infty$ and $T\in \mathcal{S}_p(\mathcal H)$,
then $T\in \mathcal{S}_q(\mathcal H)$ and
\begin{equation*}
\|T\|_q\le\|T\|_p.
\end{equation*}
\end{proposition}
\begin{proof}
 (1), (2) and (3) follow from \cite[Theorems 3.1, 2.7(a), and 2.8]{Simon}.

For (4), let $1\le p\le q<\infty$ and
$T\in \mathcal{S}_p(\mathcal H)$.
Since
\begin{equation*}
|T|^q\le\|T\|^{q-p}|T|^p,
\end{equation*}
we obtain
\begin{equation*}
\operatorname{Tr}(|T|^q)
\le\|T\|^{q-p}\operatorname{Tr}(|T|^p)
\le\|T\|_p^q.
\end{equation*}
Hence $T\in \mathcal{S}_q(\mathcal H)$ and $\|T\|_q\le\|T\|_p$.
\end{proof}

\begin{definition}
A closed subspace $\mathcal M$ is a {\em reducing null subspace} of $T$ if $\mathcal M\subseteq\ker T\cap\ker T^*$. Equivalently,
\begin{equation*}T=\begin{pmatrix}T|_{\mathcal M^\perp}&0\\0&0\end{pmatrix}\begin{matrix}\mathcal M^\perp\\\mathcal M\end{matrix}.\end{equation*}
\end{definition}
For $u,v\in\mathcal H$, write $(u\otimes v^*)h=\langle h,v\rangle u$
for every $h\in\mathcal H$.
For a unit vector $e\in\mathcal H$, let $P_e=e\otimes e^*$ be the rank one projection in $\mathcal B(\mathcal H)$ onto $\mathbb C e$.

\subsection{Regular pair}

\begin{definition}\label{def:regular}
Let $(A,B)\in M_m(\mathbb{C})\times M_m(\mathbb{C})$, and fix $p\ge1$.
We use the maximal norm
\begin{equation*}
\|(X,Y)\|_p=\max\{\|X\|_p,\|Y\|_p\}
\end{equation*}
on $M_m(\mathbb{C})\times M_m(\mathbb{C})$.
Let
\begin{equation*}
\mathcal{V}_m=\{H\in M_m(\mathbb{C}):\operatorname{Tr}H=0\}.
\end{equation*}
Define the linear map
\begin{eqnarray*}
\mathcal{L}_{A,B}:M_m(\mathbb{C})\times M_m(\mathbb{C}) &\to& \mathcal{V}_m\,. \\
  (X,Y) &\mapsto & [X,B]+[A,Y]
\end{eqnarray*}
We call $(A,B)$ \emph{a regular pair} if $\mathcal{L}_{A,B}$
is surjective.
\end{definition}

For a regular pair, the restriction of $\mathcal L_{A,B}$ to a complement of its kernel is a linear isomorphism onto $\mathcal V_m$. Its inverse therefore defines a linear right inverse
\[\mathcal R_{A,B}:\mathcal V_m\longrightarrow M_m(\mathbb C)\times M_m(\mathbb C),\qquad \mathcal L_{A,B}\mathcal R_{A,B}=I_{\mathcal V_m}.\]
 Thus, for every $H\in\mathcal V_m$, the pair $(X,Y)=\mathcal R_{A,B}H$ satisfies
\[[X,B]+[A,Y]=H,\qquad \|(X,Y)\|_p\le\|\mathcal R_{A,B}\|\|H\|_p.\]
In particular, if there exists a constant $M>0$ such that
\[\max\{\|X\|_p,\|Y\|_p\}\le M\|H\|_p\]
for every $H\in\mathcal V_m$, where $(X,Y)=\mathcal R_{A,B}H$, then $\|\mathcal R_{A,B}\|\le M$.

Let $(A,B)$ be a regular pair and $S\in M_m(\mathbb C)$ be invertible.
For $H\in\mathcal V_m$, write
\begin{equation*}
    (X,Y)=\mathcal R_{A,B}(S^{-1}HS).
\end{equation*}
It is obvious that $(SAS^{-1},SBS^{-1})$ is regular and
\[
\mathcal R_{SAS^{-1},SBS^{-1}}H
=(SXS^{-1},SYS^{-1}).
\]
Moreover, by Proposition~\ref{prop:norms},
\begin{equation*}
\|\mathcal R_{SAS^{-1},SBS^{-1}}\|
\le \|S\|^2\|S^{-1}\|^2\|\mathcal R_{A,B}\|.
\end{equation*}

\begin{proposition}\label{prop:regular}
Let $(A,B)\in M_m(\mathbb{C})\times M_m(\mathbb{C})$.
The following conditions are equivalent.
\begin{enumerate}
\item[(1)]
The pair $(A,B)$ is regular, i.e., $\mathcal{L}_{A,B}$ is surjective.

\item[(2)]
The common commutant of $A$ and $B$ consists only of scalar matrices:
\begin{equation*}
\{Z\in M_m(\mathbb{C}):ZA=AZ,\ ZB=BZ\}=\mathbb{C}I_m.
\end{equation*}

\item[(3)]
There exists a linear map
$
\mathcal{R}_{A,B}:\mathcal{V}_m\longrightarrow M_m(\mathbb{C})\times M_m(\mathbb{C})
$
such that
$$
\mathcal{L}_{A,B}\mathcal{R}_{A,B}=I_{\mathcal{V}_m}.
$$
\end{enumerate}
\end{proposition}
\begin{proof}
$(1) \Longrightarrow (2)$:
Let $Z\in M_m(\mathbb C)$ commute with both $A$ and $B$. 
Set
$
Z_0=Z-\frac{\operatorname{Tr}(Z)}m I_m.
$
Then $Z_0\in \mathcal{V}_m$ and $AZ_0=Z_0A$ and $BZ_0=Z_0B$.
Now the surjectivity of $\mathcal L_{A,B}$ yields  $(X,Y)\in M_m(\mathbb C)\times M_m(\mathbb C)$ such that
$
\mathcal L_{A,B}(X,Y)=Z_0^*.
$
Thus
\begin{equation*}
\begin{aligned}
\operatorname{Tr}(Z_0Z_0^*)
&=\operatorname{Tr}\bigl(Z_0([X,B]+[A,Y])\bigr)\\
&=\operatorname{Tr}\bigl((BZ_0-Z_0B)X\bigr)
  +\operatorname{Tr}\bigl((Z_0A-AZ_0)Y\bigr)\\
&=0.
\end{aligned}
\end{equation*}
Then $Z_0=0$, i.e., $Z$ is scalar.

$(2) \Longrightarrow (3)$:
Define a linear map
\begin{equation*}
    \Phi:\mathcal V_m\longrightarrow\mathcal V_m, \quad T \mapsto[[T,B^*],B]+[A,[A^*,T]].
\end{equation*}
For every $T\in\mathcal V_m$, a direct computation gives
\begin{equation*}
\operatorname{Tr}(T^*\Phi(T))
=\operatorname{Tr}([T,B^*]^*[T,B^*])+\operatorname{Tr}([A^*,T]^*[A^*,T])\geq 0.
\end{equation*}
Thus $\Phi(T)=0$ implies that $[T,B^*]=0$ and $[A^*,T]=0$. Then $T^*$ must be a scalar. Since  $\operatorname{Tr}(T^*)=\overline{\operatorname{Tr}(T)}=0$, we obtain $T=0$. Thus $\Phi$ is injective.

Now we define
\begin{eqnarray}
 \mathcal R_{A,B}:\mathcal{V}_m&\rightarrow & M_m(\mathbb C)\times M_m(\mathbb C) \\
  H &\mapsto & ([\Phi^{-1}(H),B^*],[A^*,\Phi^{-1}(H)])
\end{eqnarray}
Since $\Phi^{-1}$ is linear, $\mathcal R_{A,B}$ is linear and it is routine to check that
$\mathcal L_{A,B}\mathcal R_{A,B}=I_{\mathcal V_m}$. 

$(3) \Longrightarrow (1)$ is trivial.
\end{proof}

The following lemma gives a more precise characterization of trace-zero matrices.
\begin{lemma}\label{lem:zero-diagonal}
Let $0\ne D\in M_m(\mathbb C)$ have trace zero.
There exists an orthonormal basis $\{e_1,\ldots,e_m\}$
of $\mathbb C^m$ such that the matrix $(D_{ij})$ of $D$
with respect to this basis satisfies the following conditions.
\begin{enumerate}
\item[(1)] All diagonal entries are zero, i.e.,
$D_{ii}=0$ for $1\le i\le m$.
\item[(2)] For every nonempty proper subset
$J\subset\{1,\ldots,m\}$, there exist $i\in J$
and $j\notin J$ such that
\begin{equation*}
D_{ij}\ne0\quad\text{or}\quad D_{ji}\ne0.
\end{equation*}
\end{enumerate}
\end{lemma}

\begin{proof}
By \cite[Corollary~1]{Fillmore}, there is an orthonormal
basis in which $D$ has zero diagonal. Among all such bases,
choose one for which the number of nonzero entries is maximal.

We claim that this orthonormal basis satisfies the required properties. By construction, Condition (1) holds. Suppose that (2) fails. Then, for some nonempty proper
subset $J\subset\{1,\ldots,m\}$, the coordinate decomposition
associated with $J$ and $J^c$ gives
\[
 D=\begin{pmatrix}D_J&0\\0&D_{J^c}\end{pmatrix}.
\]
Since $D\ne0$, we may interchange $J$ and $J^c$ and assume
that $D_J\ne0$. Choose $i\in J$ whose row in $D_J$ is
nonzero, and choose any $j\in J^c$.

Order the coordinates as
\[
 \{i\},\quad \{j\},\quad J\setminus\{i\},
 \quad J^c\setminus\{j\}.
\]
Relative to  this ordering,
\[
 D=
 \begin{pmatrix}
 0&0&a&0\\
 0&0&0&b\\
 c&0&D_1&0\\
 0&d&0&D_2
 \end{pmatrix},
 \qquad a\ne0,
\]
where $D_1$ and $D_2$ have zero diagonal.
Let
\[
 U=\frac1{\sqrt2}
 \begin{pmatrix}1&1\\1&-1\end{pmatrix}
 \oplus I_{m-2}.
\]
Then
\[
 U^*DU=
 \begin{pmatrix}
 0&0&a/\sqrt2&b/\sqrt2\\
 0&0&a/\sqrt2&-b/\sqrt2\\
 c/\sqrt2&c/\sqrt2&D_1&0\\
 d/\sqrt2&-d/\sqrt2&0&D_2
 \end{pmatrix}.
\]
The resulting matrix retains a zero diagonal. Since $a\neq 0$, the number of non-zero entries has strictly increased, contradicting the maximality assumption. Hence, Condition (2) holds.
\end{proof}

We use the diagonal-factor construction described in \cite[Section 2]{JOS}. Here, we supplement this construction with a suitable choice of orthonormal basis to ensure that the resulting pair is regular, and choose the diagonal entries of $A_*$ to be distinct and to lie in the prescribed positive interval $[\alpha,\beta]$. The stated Schatten norm bounds then follow from elementary estimates.
\begin{lemma}\label{lem:seed}
Let $0\ne D\in M_m(\mathbb{C})$ have trace zero, $2\leq p<\infty$,
and $0<\alpha<\beta$.
There is a regular pair $(A_*,B_*)$ with $[A_*,B_*]=D$ such that,
in a suitable orthonormal basis, $A_*$ is positive diagonal with
distinct entries in $[\alpha,\beta]$, and
\begin{equation*}
\|A_*\|_p\le\beta m^{1/p},
\qquad
\|B_*\|_p\le
\frac{m-1}{\beta-\alpha}\,
m^{1/2-1/p}\|D\|_p.
\end{equation*}
\end{lemma}
\begin{proof}
Since $D\ne0$ and $\operatorname{Tr}D=0$, we have $m\ge2$.
Choose an orthonormal basis as in Lemma~\ref{lem:zero-diagonal}.
In this basis, set $\lambda_i=\alpha+(\beta-\alpha)(i-1)/(m-1)$ for $1\leq i\leq m$ and define
\begin{equation*}
A_*=\operatorname{diag}(\lambda_1,\ldots,\lambda_m),
\qquad
(B_*)_{ij}=
\begin{cases}
\dfrac{D_{ij}}{\lambda_i-\lambda_j},&i\ne j,\\[6pt]
0,&i=j.
\end{cases}
\end{equation*}
Then $A_*$ is positive diagonal with distinct entries in
$[\alpha,\beta]$ and $[A_*,B_*]=D$.

We next show that $(A_*,B_*)$ is regular.
Let $Z$ commute with both $A_*$ and $B_*$. Then $Z$ commutes with $D$ and is diagonal.
Write $Z=\operatorname{diag}(z_1, \ldots, z_m)$.
Then $ZD=DZ$ gives
\begin{equation*}
(z_i-z_j)D_{ij}=0
\qquad(1\le i,j\le m).
\end{equation*}
If $Z$ is not scalar, then
$I=\{i:z_i=z_1\}$ is a nonempty proper subset of
$\{1,\ldots,m\}$.
For every $i\in I$ and $j\notin I$, we have $z_i\ne z_j$,
and hence $D_{ij}=D_{ji}=0$.
This contradicts the property of the chosen basis.
Thus $Z$ is scalar, and Proposition~\ref{prop:regular} shows that
$(A_*,B_*)$ is regular.

Finally, since each $\lambda_i\in [\alpha,\beta]$,
\begin{equation*}
\|A_*\|_p
=\left(\sum_{i=1}^m\lambda_i^p\right)^{1/p}
\le\beta m^{1/p}.
\end{equation*}
Since $|\lambda_i-\lambda_j|\ge(\beta-\alpha)/(m-1)$
for $i\ne j$ and $2\leq p <\infty$, we obtain
\begin{equation*}
\begin{aligned}
\|B_*\|_p
&\le\|B_*\|_2
=\left(\sum_{i\ne j}
\frac{|D_{ij}|^2}{|\lambda_i-\lambda_j|^2}\right)^{1/2}\le\frac{m-1}{\beta-\alpha}\|D\|_2
\le\frac{m-1}{\beta-\alpha}\,
m^{1/2-1/p}\|D\|_p.
\end{aligned}
\end{equation*}
This completes the proof.
\end{proof}

\subsection{Norm estimates and Resolvent points}

The next lemma appends a prescribed eigenvalue and gives a right inverse for the enlarged commutator differential.

\begin{lemma}\label{lem:extension}
Let $1\le p<\infty$, $0<\theta\le1/8$, and let
$(A,B)\in M_m(\mathbb C)\times M_m(\mathbb C)$
be a regular pair with $0\neq D=[A,B]$ and
$\|\mathcal R_{A,B}\|\le L$.
Let $\lambda\notin\sigma(A)$, $R=(A-\lambda I)^{-1}$,
and choose a unit vector $x\in\mathbb C^m$ such that $Dx\ne0$.
If
\begin{equation*}
\theta\|R\|\|RDx\|\le\frac18,
\end{equation*}
then
\begin{equation*}
\widehat A:=
\begin{pmatrix}
A&\theta(A-\lambda I)x\\
0&\lambda
\end{pmatrix}
\begin{matrix}
\mathbb C^m\\
\mathbb C
\end{matrix},
\qquad
\widehat B:=
\begin{pmatrix}
B&\theta RB(A-\lambda I)x\\
0&0
\end{pmatrix}
\begin{matrix}
\mathbb C^m\\
\mathbb C
\end{matrix}
\end{equation*}
form a regular pair with
\begin{equation*}
[\widehat A,\widehat B]
=
\begin{pmatrix}
D&0\\
0&0
\end{pmatrix},
\quad
\|\mathcal R_{\widehat A,\widehat B}\|
\le
5L+\frac{3+2\|R\|\|B\|}{\theta\|RDx\|}.
\end{equation*}
\end{lemma}
\begin{proof}
Set $v=RDx$. Since $R$ is invertible and $Dx\ne0$, $v\ne0$.
Direct multiplication gives
\begin{equation*}
[\widehat A,\widehat B]
=
\begin{pmatrix}
D&\theta\bigl((A-\lambda I)R-I\bigr)B(A-\lambda I)x\\
0&0
\end{pmatrix}
=
\begin{pmatrix}
D&0\\
0&0
\end{pmatrix}.
\end{equation*}
It remains to check the norm estimate.
Set
\begin{equation*}
S=
\begin{pmatrix}
I&-\theta x\\
0&1
\end{pmatrix},
\qquad
S^{-1}=
\begin{pmatrix}
I&\theta x\\
0&1
\end{pmatrix}.
\end{equation*}
Then $\|S\|,\|S^{-1}\|\le1+\theta$.
Since $\theta RB(A-\lambda I)x-\theta Bx=-\theta v$,
\begin{equation*}
S^{-1}\widehat AS=
\begin{pmatrix}
A&0\\
0&\lambda
\end{pmatrix},
\qquad
S^{-1}\widehat BS=
\begin{pmatrix}
B&-\theta v\\
0&0
\end{pmatrix}.
\end{equation*}
Let $\widetilde A:=S^{-1}\widehat AS$ and
$\widetilde B:=S^{-1}\widehat BS$.

We first construct a linear right inverse of
$\mathcal L_{\widetilde A,\widetilde B}$.
Let
\begin{equation*}
H=
\begin{pmatrix}
H_0&h\\
k&\gamma
\end{pmatrix}\in \mathcal{V}_{m+1},
\qquad
w:=-\frac{\gamma v^*}{\theta\|v\|^2}.
\end{equation*}
Then
\begin{equation*}
-\theta wv=\gamma,
\qquad
\operatorname{Tr}(H_0-\theta v\,w)
=-\gamma-\theta wv=0.
\end{equation*}
Thus we may set
\begin{equation*}
(X_0,Y_0)=\mathcal R_{A,B}(H_0-\theta v\,w).
\end{equation*}
Define
\begin{equation*}
X:=
\begin{pmatrix}
X_0&0\\
w&0
\end{pmatrix},
\qquad
Y:=
\begin{pmatrix}
Y_0&R(h+\theta X_0v)\\
(wB-k)R&0
\end{pmatrix}.
\end{equation*}
Direct computation gives
\begin{equation*}
\begin{aligned}
\mathcal L_{\widetilde A,\widetilde B}(X,Y)
&=[X,\widetilde B]+[\widetilde A,Y]\\
&=
\begin{pmatrix}
[X_0,B]+[A,Y_0]+\theta v\,w&h\\
k&-\theta wv
\end{pmatrix}
=H.
\end{aligned}
\end{equation*}
Hence $\mathcal L_{\widetilde A,\widetilde B}$ is surjective. Since $w$ depends linearly on $\gamma$, $\mathcal R_{A,B}$ is linear and  all the blocks of $X$ and $Y$
depend linearly on $H$, 
\begin{eqnarray*}
 \mathcal R_{\widetilde A,\widetilde B}:\mathcal{V}_{m+1}&\to& M_{m+1}(\mathbb C)\times M_{m+1}(\mathbb C) \\
  H &\mapsto & (X,Y)
\end{eqnarray*}
is a linear right inverse of
$\mathcal L_{\widetilde A,\widetilde B}$.
Thus $(\widetilde A,\widetilde B)$ is regular by Proposition~\ref{prop:regular}.

We next estimate the norm of
$\mathcal R_{\widetilde A,\widetilde B}$.
Since
\begin{equation*}
\|H_0\|_p,\ \|h\|,\ \|k\|,|\gamma|
\le\|H\|_p,
\qquad
\|w\|
=\frac{|\gamma|}{\theta\|v\|}
\le\frac{\|H\|_p}{\theta\|v\|}.
\end{equation*}
Consequently,
\begin{equation*}
\|(X_0,Y_0)\|_p
\le\|\mathcal R_{A,B}\|\,\|H_0-\theta v\,w\|_p
\le2L\|H\|_p.
\end{equation*}
Hence
\begin{equation*}
\|X\|_p
\le\|X_0\|_p+\|w\|
\le
\left(2L+\frac{1}{\theta\|v\|}\right)\|H\|_p,
\end{equation*}
\begin{equation*}
\begin{aligned}
\|Y\|_p
&\le
\|Y_0\|_p
+\|R\|\bigl(\|h\|+\theta\|v\|\|X_0\|_p\bigr)
+\bigl(\|w\|\|B\|+\|k\|\bigr)\|R\|\\
&\le
\left(
2L(1+\theta\|R\|\|v\|)
+2\|R\|
+\frac{\|R\|\|B\|}{\theta\|v\|}
\right)\|H\|_p.
\end{aligned}
\end{equation*}
It follows from $\theta\|R\|\|v\|\le\frac18$ that
\begin{equation*}
\begin{aligned}
\|\mathcal R_{\widetilde A,\widetilde B}\|
&\le
\max\left\{
2L+\frac{1}{\theta\|v\|},\
2L(1+\theta\|R\|\|v\|)
+2\|R\|
+\frac{\|R\|\|B\|}{\theta\|v\|}
\right\}\\
&\le
2L(1+\theta\|R\|\|v\|)
+\frac{1+\|R\|\|B\|+2\theta\|R\|\|v\|}
{\theta\|v\|}\\
&\le
\frac94L+\frac{5/4+\|R\|\|B\|}{\theta\|v\|}.
\end{aligned}
\end{equation*}

Finally, Proposition~\ref{prop:regular} shows that $(\widehat A,\widehat B)$
is regular and $\mathcal L_{\widehat A,\widehat B}$
admits a linear right inverse satisfying
\begin{equation*}
\begin{aligned}
\|\mathcal R_{\widehat A,\widehat B}\|
&\le
\|S\|^2\|S^{-1}\|^2
\|\mathcal R_{\widetilde A,\widetilde B}\|\\
&\le
(1+\theta)^4
\left(
\frac94L+\frac{5/4+\|R\|\|B\|}{\theta\|v\|}
\right)\\
&\le
5L+\frac{3+2\|R\|\|B\|}{\theta\|v\|}.\\
\end{aligned}
\end{equation*}
The third estimate uses  $(1+\theta)^4\leq 2$.
This completes the proof.
\end{proof}

\begin{lemma}\label{lem:correction}
Let $1\le p<\infty$ and $L>0$, and let
$(A,B)\in M_m(\mathbb{C})\times M_m(\mathbb{C})$
be a regular pair with $\|\mathcal{R}_{A,B}\|\le L$.
Suppose $H\in M_m(\mathbb{C})$, $\operatorname{Tr}H=0$, and
\begin{equation*}
\|H\|_p\le\frac{1}{16L^2}.
\end{equation*}
Then there exists $(X,Y)\in M_m(\mathbb{C})\times M_m(\mathbb{C})$ with
$\|(X,Y)\|_p\le 2L\|H\|_p$ such that
\begin{equation*}
[A+X,B+Y]=[A,B]+H.
\end{equation*}
Moreover, $(A+X,B+Y)$ is regular and $\|\mathcal{R}_{A+X,B+Y}\|\leq 2L$.
\end{lemma}
\begin{proof}
If $H=0$, take $X=Y=0$. Otherwise, let $\eta=2L\|H\|_p$.
Define
\begin{eqnarray*}
\Psi:M_m(\mathbb{C})\times M_m(\mathbb{C}) &\to& M_m(\mathbb{C})\times M_m(\mathbb{C}).\\
 (U,V) &\mapsto& \mathcal{R}_{A,B}(H-[U,V])
\end{eqnarray*}
This is well defined since $\operatorname{Tr}(H-[U,V])=0$.
Let $\mathcal{B}_{\eta}$ be the closed ball of radius $\eta$,
centred at $(0,0)$.
For $(U,V)\in\mathcal{B}_{\eta}$, Proposition~\ref{prop:norms} gives
\begin{equation*}
\begin{aligned}
\|\Psi(U,V)\|_p
&\le \|\mathcal{R}_{A,B}\|\,\|H-[U,V]\|_p\le L\|H\|_p+2L\eta^2
\le\frac{3\eta}{4}.
\end{aligned}
\end{equation*}
Thus $\Psi(\mathcal{B}_{\eta})\subseteq\mathcal{B}_{\eta}$.
For $(U,V), (U',V')$ in $\mathcal{B}_{\eta}$, Proposition~\ref{prop:norms} gives
\begin{equation*}
\begin{aligned}
\|\Psi(U,V)-\Psi(U',V')\|_p
&\le L\bigl(
\|[U-U',V]\|_p+\|[U',V-V']\|_p
\bigr)\\
&\le 2L\eta\bigl(
\|U-U'\|_p+\|V-V'\|_p
\bigr)\\
&\le 4L\eta\|(U-U',V-V')\|_p.
\end{aligned}
\end{equation*}
Since $4L\eta\le 1/2$, $\Psi$ is a contraction on the
complete metric space $\mathcal{B}_{\eta}$.
By the contraction mapping theorem, there exists $(X,Y)\in\mathcal{B}_{\eta}$ such that
\begin{equation*}
    (X,Y)=\mathcal{R}_{A,B}(H-[X,Y]).
\end{equation*}
It follows that $[X,B]+[A,Y]=H-[X,Y]$, and hence $[A+X,B+Y]=[A,B]+H$.

Next we construct a linear right inverse of
$\mathcal{L}_{A+X,B+Y}$ satisfying the required norm bound.
Note that $(X,Y)\in\mathcal{B}_{\eta}$. Hence,
\begin{equation*}
\begin{aligned}
\|\mathcal{L}_{A+X,B+Y}-\mathcal{L}_{A,B}\|
&=\sup_{\|(U,V)\|_p\le 1}
\|[U,Y]+[X,V]\|_p\\
&\le\sup_{\|(U,V)\|_p\le 1}
\bigl(
2\|U\|_p\|Y\|_p+2\|X\|_p\|V\|_p
\bigr)\\
&\le 2(\|X\|_p+\|Y\|_p)
\le 4\eta.
\end{aligned}
\end{equation*}
It follows that
\begin{equation*}
\|\mathcal{L}_{A+X,B+Y}\mathcal{R}_{A,B}-I_{\mathcal{V}_m}\|
\le
\|\mathcal{L}_{A+X,B+Y}-\mathcal{L}_{A,B}\|\,
\|\mathcal{R}_{A,B}\|
\le 4L\eta\le\frac{1}{2}.
\end{equation*}
The Neumann series shows that
$\mathcal{L}_{A+X,B+Y}\mathcal{R}_{A,B}$
is invertible with
\begin{equation*}
\bigl\|
(\mathcal{L}_{A+X,B+Y}\mathcal{R}_{A,B})^{-1}
\bigr\|
\leq 2.
\end{equation*}
Thus
\begin{equation*}
\mathcal{R}_{A+X,B+Y}
:=\mathcal{R}_{A,B}
\bigl(
\mathcal{L}_{A+X,B+Y}\mathcal{R}_{A,B}
\bigr)^{-1}.
\end{equation*}
is a linear right inverse of
$\mathcal{L}_{A+X,B+Y}$, and
\begin{equation*}
\|\mathcal{R}_{A+X,B+Y}\|\le \|\mathcal{R}_{A,B}\|\,
\bigl\|
(\mathcal{L}_{A+X,B+Y}\mathcal{R}_{A,B})^{-1}
\bigr\|\le 2L.
\end{equation*}
Moreover, by Proposition~\ref{prop:regular}, the pair $(A+X,B+Y)$ is regular.
\end{proof}

We choose a point on each prescribed circle by a finite-rank approximation and a finite-dimensional determinant.
\begin{lemma}\label{lem:circle}
Let $1\le p<\infty$ and $M\ge1$.
Set $c=3(2M)^p$. For any operator $C \in \mathcal{S}_p(\mathcal{H})$ satisfying $\Vert{}C\Vert{} \le 1$ and $\Vert{}C\Vert{}_p \le M$, there exists a sequence $(\lambda_n)_{n \ge 1} \subset \mathbb{C} \setminus \sigma(C)$ with
$|\lambda_n|=n^{-2/p}$ such that
\begin{equation*}
\|(C-\lambda_nI)^{-1}\|\le e^{cn^3}
\qquad(n\ge1).
\end{equation*}
\end{lemma}
\begin{proof}
Fix $n\ge1$, and set $\rho=n^{-2/p}$. Define
\begin{equation*}
P=\chi_{(\rho/2,\infty)}(|C|),
\qquad F=CP,\qquad R=C(I-P).
\end{equation*}
Then $C=F+R$ and
\begin{equation*}
|C|^p\ge(\rho/2)^pP, \qquad
\|F\|\leq 1,\qquad
\|R\|=\bigl\||C|(I-P)\bigr\|\le\frac{\rho}{2}.
\end{equation*}
Consequently,
\begin{equation*}
\operatorname{Tr}(P)
\le\left(\frac2\rho\right)^p\operatorname{Tr}(|C|^p)
\le\left(\frac{2M}{\rho}\right)^p<\infty.
\end{equation*}
Since $\operatorname{Tr}(P)=\operatorname{rank}P$, $P$ has finite rank.
Hence
\begin{equation*}
d:=\operatorname{rank}F
\le\operatorname{rank}P
\le\left(\frac{2M}{\rho}\right)^p
=(2M)^pn^2.
\end{equation*}

If $d=0$, then $C=R$. For any $|\lambda|=\rho$,
the Neumann series  gives
\begin{equation*}
\|(C-\lambda I)^{-1}\|
=\frac1\rho
\left\|\sum_{k=0}^{\infty}\left(\frac R\rho\right)^k\right\|
\le\frac1\rho\sum_{k=0}^{\infty}\frac{1}{2^{k}}
=\frac2\rho.
\end{equation*}

Suppose that $d\ge1$, and let $\mathcal E:=\operatorname{ran}F$.
For $|z|<2/\rho$, we have $\|zR\|<1$, so the Neumann series
implies that $I-zR$ is invertible. Define operator-valued functions
\begin{equation*}
K(z)=zF(I-zR)^{-1},
\qquad
D(z)=I_{\mathcal E}-K(z)|_{\mathcal E}.
\end{equation*}
Then $K(z)$ and $D(z)$ are well defined because
$I-zR$ is invertible and $\operatorname{ran}K(z)\subseteq\mathcal E$.
Direct multiplication gives
$$
I-zC=(I-K(z))(I-zR).
$$
The Neumann series shows that the operator-valued functions $K(z)$
and $D(z)$ are analytic on $|z|<2/\rho$.
Since $\mathcal E$ is finite dimensional, $\det D(z)$ is analytic
there as well. For $|z|\le1/\rho$,
$$
\|(I-zR)^{-1}\|\le2,\qquad
\|K(z)\|\le\frac2\rho,\qquad
\|D(z)\|\le1+\frac2\rho.
$$
Using the Maximum Modulus Principle
and $\det D(0)=1$, there exists a $z_0$ such that
\begin{equation*}
|z_0|=\frac1\rho,
\qquad
|\det D(z_0)|\ge1.
\end{equation*}
Thus $D(z_0)$ is invertible. Given any unit vector
$x_1\in\mathcal E$, extend it to an orthonormal basis
$x_1,\ldots,x_d$ of $\mathcal E$.
Hadamard's inequality \cite[Theorem 7.8.1]{Horn-Johnson} gives
\begin{equation*}
1\le|\det D(z_0)|
\le\prod_{j=1}^{d}\|D(z_0)x_j\|
\le\|D(z_0)x_1\|\,\|D(z_0)\|^{d-1}.
\end{equation*}
By the arbitrariness of $x_1$, for any unit vector $v\in\mathcal E$, we obtain
\begin{equation*}
\|D(z_0)^{-1}v\|\le\|D(z_0)\|^{d-1}\|v\|.
\end{equation*}
Taking the supremum over all such $v$, we have
$$
\|D(z_0)^{-1}\|
\le\|D(z_0)\|^{d-1}
\le\left(1+\frac2\rho\right)^{d-1}.
$$

Since $\operatorname{ran}K(z_0)\subseteq\mathcal E$ and
$D(z_0)=(I-K(z_0))|_{\mathcal E}$,
$D(z_0)^{-1}K(z_0)$ can be viewed as an operator on $\mathcal H$,
and direct multiplication gives
$$
(I-K(z_0))^{-1}=I+D(z_0)^{-1}K(z_0).
$$
Now we obtain
\begin{equation*}
(I-z_0C)^{-1}
=(I-z_0R)^{-1}\bigl(I+D(z_0)^{-1}K(z_0)\bigr).
\end{equation*}
Now let $\lambda_n=z_0^{-1}$. Thus $\lambda_n\notin\sigma(C)$ and
$|\lambda_n|=\rho$. It follows  that
\begin{equation*}
\begin{aligned}
\|(C-\lambda_nI)^{-1}\|
&=|z_0|\,\|(I-z_0C)^{-1}\|\\
&\le |z_0|\,\|(I-z_0R)^{-1}\|\,
\|I+D(z_0)^{-1}K(z_0)\|\\
&\le\frac2\rho
\bigl(1+\|D(z_0)^{-1}\|\,\|K(z_0)\|\bigr)\\
&\le\frac2\rho
\left[1+\frac2\rho\left(1+\frac2\rho\right)^{d-1}\right]\le\frac2\rho\left(1+\frac2\rho\right)^d.
\end{aligned}
\end{equation*}

Finally, by applying $d\leq (2M)^pn^2$, $\ln(2n^2)\le n^2$ and $\ln(1+2n^2)\le2n$, we obtain
\begin{equation*}
\begin{aligned}
\ln\!\left[\frac{2}{\rho}\left(1+\frac{2}{\rho}\right)^d\right]
&=\ln(2n^{2/p})+d\ln(1+2n^{2/p})\\
&\le\ln(2n^2)+(2M)^pn^2\ln(1+2n^2)\\
&\le n^2+2(2M)^pn^3\\
&\le3(2M)^pn^3.
\end{aligned}
\end{equation*}
Hence $c=3(2M)^p$ has the required property.
\end{proof}

\subsection{Anderson's construction}
We adapt Anderson's construction \cite{Anderson};
see also \cite[proof of Theorem 3.1]{BPW}. While almost all of the underlying conclusions are already known, we focus here on providing explicit simultaneous operator and Schatten norm estimates.
\begin{theorem}\label{lem:stable}
Let $1\le p,q<\infty$ and $2<r<\infty$ satisfy
$r^{-1}=p^{-1}+q^{-1}$.
If $0\neq D\in \mathcal{S}_r(\mathcal H)$ has an infinite-dimensional reducing
null subspace, then there exist $C\in \mathcal{S}_p(\mathcal H)$ and
$Z\in \mathcal{S}_q(\mathcal H)$ such that $D=[C,Z]$ and
\begin{equation*}
\|C\|\le2\|D\|^{r/p},
\qquad
\|Z\|\le2\|D\|^{r/q},
\end{equation*}
and
\begin{equation*}
\|C\|_p\le2\zeta(r-1)^{1/p}\|D\|_r^{r/p},
\qquad
\|Z\|_q\le2\zeta(r-1)^{1/q}\|D\|_r^{r/q}.
\end{equation*}
\end{theorem}
\begin{proof}
Write
\begin{equation*}
D=\begin{pmatrix}D_0&0\\
0&0
\end{pmatrix}
\begin{matrix}
\mathcal M\\
\mathcal M^\perp
\end{matrix},
\end{equation*}
where $\mathcal M\ne\{0\}$ and $\mathcal M^\perp$ is
infinite-dimensional. Set $\alpha=r/p$, $\beta=r/q,$ so $\alpha+\beta=1.
$

We use Anderson's block construction \cite{Anderson}, with weights
adapted to $p$ and $q$.  Starting from $D_0$, we construct sequences $(D_k)_{k\in\mathbb Z}$,
$(U_k)_{k\in\mathbb Z}$, and $(V_k)_{k\in\mathbb Z}$
by induction in both directions.
Once $D_k$ has been defined, write $D_k=W_k|D_k|$ for its polar decomposition.

For $k\geq0$, define
\begin{equation*}
U_{k+1}=W_k|D_k|^\alpha,
\qquad
V_k=|D_k|^\beta,
\qquad
D_{k+1}=V_kU_{k+1}.
\end{equation*}
For $k\leq0$, define
\begin{equation*}
U_k=|D_k|^\alpha,
\qquad
V_{k-1}=W_k|D_k|^\beta,
\qquad
D_{k-1}=U_kV_{k-1}.
\end{equation*}
Then
\begin{equation*}
U_{k+1}V_k=V_{k-1}U_k=D_k
\qquad(k\in\mathbb Z).
\end{equation*}
The operator norm inequality and H\"older's inequality give
$\|D_k\|\le\|D\|$ and $\|D_k\|_r\le\|D\|_r$. Consequently,
\[
\begin{aligned}
\|U_k\|&\le\|D\|^\alpha,
&\qquad \|U_k\|_p&\le\|D\|_r^\alpha,\\
\|V_k\|&\le\|D\|^\beta,
&\qquad \|V_k\|_q&\le\|D\|_r^\beta.
\end{aligned}
\]

Following Anderson's notation, let $A_n,X_n$ be
$n\times(n+1)$ operator matrices and $B_n,Y_n$ be
$(n+1)\times n$ operator matrices. Their only nonzero entries are
\[
\begin{aligned}
(A_n)_{j,j}
 =\frac{(n+1-j)^\beta}{n}U_{n+2-2j},\quad
&
(B_n)_{j+1,j}
 =-\frac{j^\beta}{n+1}U_{n+1-2j},\\
(X_n)_{j,j+1}
 =\frac{j^\alpha}{n}V_{n-2j},\qquad\qquad
&
(Y_n)_{j,j}
 =\frac{(n+1-j)^\alpha}{n+1}V_{n+1-2j},
\end{aligned}
\quad 1\le j\le n.
\]
On $\bigoplus_{n\ge1}\mathcal M^{\oplus n}$, set
\[
C=
\begin{pmatrix}
0&A_1&0&\cdots\\
B_1&0&A_2&\ddots\\
0&B_2&0&\ddots\\
\vdots&\ddots&\ddots&\ddots
\end{pmatrix},
\qquad
Z=
\begin{pmatrix}
0&X_1&0&\cdots\\
Y_1&0&X_2&\ddots\\
0&Y_2&0&\ddots\\
\vdots&\ddots&\ddots&\ddots
\end{pmatrix}.
\]
Anderson's block calculation remains valid without modification, as the above relations together with $\alpha + \beta = 1$ ensure that
\[
[C,Z]=\begin{pmatrix}D_0&0\\
0&0
\end{pmatrix}.
\]
Since $\bigoplus_{n\ge2}\mathcal M^{\oplus n}
\cong\mathcal M^\perp$, a unitary fixing the first summand
$\mathcal M$ carries $D_0\oplus0$ to $D$.
Conjugating $C$ and $Z$ by this unitary gives the required
factors on $\mathcal H$, with all norm estimates preserved.

It remains to verify the norm estimates.
Each of $A_n,B_n,X_n,Y_n$ has at most $n$ nonzero
operator entries, in distinct rows and columns. Hence
\[
\max\{\|A_n\|,\|B_n\|\}
\le n^{-\alpha}\|D\|^\alpha,
\qquad
\max\{\|X_n\|,\|Y_n\|\}
\le n^{-\beta}\|D\|^\beta.
\]
Taking suprema over $n$ for the upper and lower block
diagonals shows that $C$ and $Z$ define bounded operators, with
\[
\|C\|\le 2\|D\|^\alpha,
\qquad
\|Z\|\le 2\|D\|^\beta.
\]

Since the nonzero operator entries occupy distinct rows
and columns, and $\alpha p=\beta q=r$, we have
\[
\|C\|_p
\le
(\sum_{n\ge1}\|A_n\|_p^p)^{1/p}
+
(\sum_{n\ge1}\|B_n\|_p^p)^{1/p}
\le2\zeta(r-1)^{1/p}\|D\|_r^\alpha.
\]
An identical argument yields
\[
\|Z\|_q \le 2 \zeta(r-1)^{1/q} \|D\|_r^\beta.
\]
Since $r > 2$, these series converge. Transporting $C$ and $Z$ back to $\mathcal{H}$ via the preceding unitary equivalence completes the proof.
\end{proof}

\section{Localized compressions and trace adjustment}\label{sec:geometry}

In this section, suppose that $T$ is a compact operator of infinite rank. We then construct a similarity such that one compression admits an infinite-dimensional reducing null subspace, and the complementary compression is trace-class with zero trace.  We use the finite-codimension selection in
\cite[Lemma 3.1]{Liu}, followed by a rank-one shear to adjust the
trace, while controlling the similarity and the off-diagonal blocks.

\begin{lemma}\label{lem:vector-estimate}
Let $T\in\mathcal B(\mathcal H)$ and $0<\delta\le1/4$.
Suppose that $x,y$ are unit vectors satisfying
\[
 Tx\ne0,\qquad \langle x,y\rangle=\langle Tx,y\rangle=0,
 \qquad
 \|Ty\|\le\delta^2\|Tx\|.
\]
Let
\[
 v=\frac{\delta x+y}{\sqrt{1+\delta^2}},
 \qquad
 \lambda=\langle Tv,v\rangle,
 \qquad
 z=Tv- \lambda v.
\]
Then
\[
 | \lambda|\le2\delta^2\|Tx\|,
 \qquad
 \frac{\delta}{2}\|Tx\|
 \le\|z\|\le\|Tv\|\leq2\delta\|Tx\|.
\]
\end{lemma}

\begin{proof}
Since $x,y$ are orthonormal, $\|v\|=1$ and
$\langle z,v\rangle=0$. The triangle inequality gives
\[
 \begin{aligned}
 \|Tv\|
 &\ge
 \frac{\delta\|Tx\|-\|Ty\|}{\sqrt{1+\delta^2}}
 \ge
 \frac{\delta(1-\delta)}{\sqrt{1+\delta^2}}\|Tx\|
 \ge\frac{1}{2}\delta\|Tx\|,\\
 \|Tv\|
 &\le
 \frac{\delta\|Tx\|+\|Ty\|}{\sqrt{1+\delta^2}}
 \le
 \frac{\delta(1+\delta)}{\sqrt{1+\delta^2}}\|Tx\|
 \leq2\delta\|Tx\|.
 \end{aligned}
\]
Here we used
\[
 \frac{1-\delta}{\sqrt{1+\delta^2}}
 \ge\frac{1-\delta}{1+\delta}
 \ge1-2\delta\ge\frac{1}{2}.
\]
Since $\langle Tx,y\rangle=0$, we also have
\[
 \begin{aligned}
 |\lambda|
 &=
 \left|
 \frac{\delta^2}{1+\delta^2}\langle Tx,x\rangle
 +\frac{\langle Ty,v\rangle}{\sqrt{1+\delta^2}}
 \right|\\
 &\le
 \frac{\delta^2}{1+\delta^2}\|Tx\|
 +\frac{\|Ty\|}{\sqrt{1+\delta^2}}\\
 &\le2\delta^2\|Tx\|.
 \end{aligned}
\]
Consequently,
\[
 \|z\|^2
 =\|Tv\|^2-|\lambda|^2
 \ge
 \delta^2\left(\frac{9}{17}-4\delta^2\right)\|Tx\|^2
 \ge\frac{\delta^2}{4}\|Tx\|^2.
\]
Then we have
\[
 \frac{\delta}{2}\|Tx\|
 \le\|z\|\le\|Tv\|\leq2\delta\|Tx\|.
\]
\end{proof}

\begin{proposition}\label{prop:form}
Let $T$ be a compact operator of infinite rank, let $0<\theta<1/4$,
and let $(\varepsilon_n)$ be a positive sequence decreasing to zero.
Given a finite-dimensional subspace $\mathcal M\subset\mathcal H$,
there exist a bounded invertible operator $S_0$, an orthogonal
decomposition
\[
 \mathcal H=\mathcal H_0\oplus\bigoplus_{n\ge1}\mathcal H_{1,n},
 \qquad
 1\le\dim\mathcal H_{1,n}\le2,
\]
and an orthonormal sequence
$(f_n)\subset\mathcal H_0\cap\mathcal M^\perp$ such that
\[
 \widetilde T=S_0^{-1}TS_0
 =
 \begin{pmatrix}
 D&E_1&E_2&\cdots\\
 F_1&G_1&0&\cdots\\
 F_2&0&G_2&\cdots\\
 \vdots&\vdots&\vdots&\ddots
 \end{pmatrix}
 \begin{matrix}
   \mathcal H_{0} \\
   \mathcal H_{1,1} \\
   \mathcal H_{1,2} \\
   \vdots
 \end{matrix} \,,
\]
where
\[
 \begin{gathered}
 \widetilde Tf_n,\ \widetilde T^*f_n\in\mathcal H_{1,n},
 \qquad Df_n=D^*f_n=0,\\
 \max\{\|E_n\|,\|F_n\|,\|G_n\|\}\le2\varepsilon_n,
 \qquad
 \max\{\|S_0\|,\|S_0^{-1}\|\}<1+\theta.
 \end{gathered}
\]
Moreover, $\mathcal H_{1,n}\subset\mathcal M^\perp$, and
$S_0$ and $S_0^*$ act identically on $\mathcal M$.

\end{proposition}

\begin{proof}
Fix $\delta=\min\{1/4,\theta/4\}$. Since $T$ is compact, the image
of the unit ball is totally bounded. For each $n$, there is a
finite-dimensional subspace $\mathcal Y_n$
 such that
\[
\operatorname{dist}(Th,\mathcal Y_n)
 \le\varepsilon_n,\quad {\rm for\,\,all}\,\, \|h\|\le1.
\]
Set $R_n=P_{\mathcal Y_n}T$.
Then $R_n$ has finite rank and
$\|T-R_n\|\le\varepsilon_n$.

Let
\[
 \mathcal F_n
 =
 \mathcal M+\operatorname{ran}R_n+\operatorname{ran}R_n^*.
\]
For any $h\in\mathcal F_n^\perp$, we have $R_nh=R_n^*h=0$.
Consequently,
\[
\|Th\|\le\varepsilon_n\|h\|,
 \qquad \|T^*h\|\le\varepsilon_n\|h\|.
\]

Now we construct unit vectors $f_n$ and finite-dimensional
subspaces $\mathcal V_n$ inductively.
For $n=1$, set
\[
 \mathcal E_1=\mathcal F_1,
 \qquad
 \mathcal L_1
 =\mathcal E_1+T\mathcal E_1+T^*\mathcal E_1
                 +T^*T\mathcal E_1.
\]
Since $\mathcal L_1$ is finite-dimensional and $T$ has
infinite rank, there is a unit vector
$x_1\in\mathcal L_1^\perp$ with $Tx_1\ne0$.
Otherwise $T=TP_{\mathcal L_1}$ would have finite rank.

Since every orthonormal sequence in $\bigl(\mathcal L_1+ \operatorname{span}\{x_1,Tx_1\}\bigr)^\perp$ converges weakly to zero, the compactness of $T$ guarantees that its image sequence converges to zero in norm. Consequently, one can find a unit vector $y_1\in\bigl(\mathcal L_1+
       \operatorname{span}\{x_1,Tx_1\}\bigr)^\perp$  such that
\[
 \|Ty_1\|\le\delta^2\|Tx_1\|.
\]
Now we set
\[
 f_1=\frac{\delta x_1+y_1}{\sqrt{1+\delta^2}},
 \qquad
 \mathcal V_1
 =\operatorname{span}\{f_1,Tf_1,T^*f_1,T^*Tf_1\}.
\]

For $n\ge2$, suppose that
$x_j,y_j,f_j,\mathcal V_j$ have been constructed
for $1\le j<n$. Let
\[
 \begin{aligned}
 \mathcal E_n
 &=\mathcal F_n+
   \sum_{j=1}^{n-1}
   \bigl(\mathcal V_j+T\mathcal V_j+T^*\mathcal V_j\bigr),\\
 \mathcal L_n
 &=\mathcal E_n+T\mathcal E_n+T^*\mathcal E_n
                   +T^*T\mathcal E_n.
 \end{aligned}
\]
Both spaces are finite-dimensional. Using the compactness of $T$ again, there are unit vectors $x_n,y_n$ satisfying
\[
 \begin{gathered}
 x_n\in \mathcal L_n^\perp,\qquad Tx_n\ne0,\\
 y_n\in \big(\mathcal L_n+
             \operatorname{span}\{x_n,Tx_n\}\big)^\perp,
 \quad
 \|Ty_n\|\le\delta^2\|Tx_n\|.
 \end{gathered}
\]
Set
\[
 f_n=\frac{\delta x_n+y_n}{\sqrt{1+\delta^2}},
 \qquad
 \mathcal V_n
 =\operatorname{span}\{f_n,Tf_n,T^*f_n,T^*Tf_n\}.
\]
Then  we have
\[
 \|f_n\|=1,\quad f_n\in\mathcal L_n^\perp,
 \quad \dim\mathcal V_n\le4.
\]
This completes the inductive construction.

From the construction, we have
\[
 \langle f_n,h\rangle=0,\qquad
 \langle Tf_n,h\rangle
   =\langle f_n,T^*h\rangle=0,
\]
\[
 \langle T^*f_n,h\rangle
   =\langle f_n,Th\rangle=0,\qquad
 \langle T^*Tf_n,h\rangle
 =\langle f_n,T^*Th\rangle=0,
\]
for every $h\in\mathcal E_n$. Hence $\mathcal V_n\subset\mathcal E_n^\perp$.
In particular,
\[\mathcal V_n\subset\mathcal F_n^\perp,\qquad
\mathcal V_n\subset \big(\sum_{j=1}^{n-1}
   \left(\mathcal V_j+T\mathcal V_j+T^*\mathcal V_j\right)\big)^\perp.
\]
It follows that the spaces $\mathcal V_n$ are mutually
orthogonal and
\[
 P_{\mathcal V_m}T|_{\mathcal V_n}=0,
 \qquad m,n\ge1,\ m\ne n.
\]
Moreover,
\[
 \max\{\|T|_{\mathcal V_n}\|,
          \|T^*|_{\mathcal V_n}\|\}\leq \max\{\|T|_{\mathcal F_n^\perp}\|,
          \|T^*|_{\mathcal F_n^\perp}\|\}\leq\varepsilon_n.
\]
Denote
$
 \mathcal V_0=
 \left(\bigoplus_{n\ge1}\mathcal V_n\right)^\perp$ and write $T_{ij}=P_{\mathcal V_i}T|_{\mathcal V_j}$.
Relative to
$
 \mathcal H=\mathcal V_0\oplus\mathcal V_1
                         \oplus\mathcal V_2\oplus\cdots,
$
$T$ has the form
\[
 T=
 \begin{pmatrix}
 T_{00}&T_{01}&T_{02}&\cdots\\
 T_{10}&T_{11}&0&\cdots\\
 T_{20}&0&T_{22}&\cdots\\
 \vdots&\vdots&\vdots&\ddots
 \end{pmatrix}\begin{array}{c}
                \mathcal V_0 \\
                \mathcal V_1 \\
                \mathcal V_2 \\
                \vdots
              \end{array}.
\]

We next construct a small nilpotent perturbation on each
$\mathcal V_n$. For each $n$, set
\[
\lambda_n=\langle Tf_n,f_n\rangle,
 \qquad
 z_n=Tf_n-\lambda_nf_n.
\]
Applying Lemma~\ref{lem:vector-estimate} to $x_n,y_n$
gives
\[
 |\lambda_n|\le2\delta^2\|Tx_n\|,
 \qquad
 \|z_n\|\ge\frac{\delta}{2}\|Tx_n\|>0.
\]

Define
\[
 N_n=-\frac{\lambda_n}{\|z_n\|^2}\,f_n\otimes z_n^*.
\]
For every $u\in\mathcal H$, the rank-one formula gives
\[
 N_nu=-\frac{\lambda_n}{\|z_n\|^2}\langle u,z_n\rangle f_n,
 \qquad
 N_n^*u=-\frac{\overline{\lambda_n}}{\|z_n\|^2}
              \langle u,f_n\rangle z_n.
\]
It follows that
\[
 \operatorname{ran}N_n
 \subseteq\mathbb C f_n\subseteq\mathcal V_n,
 \qquad
 \operatorname{ran}N_n^*
 \subseteq\mathbb C z_n\subseteq\mathcal V_n
\]
and
\[
\|N_n\|=\frac{|\lambda_n|}{\|z_n\|}\le4\delta.
\]
Moreover,
\[
 N_nf_n=0,\qquad
 N_nTf_n=-\lambda_nf_n,\qquad
 N_n^2=0.
\]
Hence $(I-N_n)^{-1}=I+N_n$ and
\[
 (I+N_n)T(I-N_n)f_n
 =Tf_n-\lambda_nf_n=z_n.
\]

Denote
\[
 N=0_{\mathcal V_0}\oplus\bigoplus_{n\ge1}N_n,
 \qquad
 S_0=I-N.
\]
Then $N^2=0$, $S_0^{-1}=I+N$ and
\[
 \|N\|=\sup_n\|N_n\|\le4\delta,
\qquad \max\{\|S_0\|,\|S_0^{-1}\|\}
 <1+4\delta
 \leq 1+\theta.
\]
Since every $\mathcal V_n$ is orthogonal to $\mathcal M$,
both $N$ and $N^*$ vanish on $\mathcal M$.
Thus $S_0$ and $S_0^*$ act identically on $\mathcal M$.

Let \[
\widetilde T=S_0^{-1}TS_0.
\]
Since $N_n=P_{\mathcal V_n}N_nP_{\mathcal V_n}$,
$\mathcal V_n$ reduces $N$, $S_0$, and $S_0^{-1}$.
Thus, for any $m,n\geq 1$, with $T_{ij}=P_{\mathcal V_i}T|_{\mathcal V_j}$,
block multiplication gives
\[
 \begin{aligned}
 P_{\mathcal V_0}\widetilde T|_{\mathcal V_0}
 &=T_{00},\\
 P_{\mathcal V_0}\widetilde T|_{\mathcal V_n}
 &=T_{0n}(I-N_n),\\
 P_{\mathcal V_n}\widetilde T|_{\mathcal V_0}
 &=(I+N_n)T_{n0},\\
 P_{\mathcal V_m}\widetilde T|_{\mathcal V_n}
 &=(I+N_m)T_{mn}(I-N_n).
 \end{aligned}
\]
Consequently,
$\widetilde T$ has the form
\[
 \begin{pmatrix}
 T_{00}
   &T_{01}(I-N_1)
   &T_{02}(I-N_2)
   &\cdots\\
 (I+N_1)T_{10}
   &(I+N_1)T_{11}(I-N_1)
   &0
   &\cdots\\
 (I+N_2)T_{20}
   &0
   &(I+N_2)T_{22}(I-N_2)
   &\cdots\\
 \vdots&\vdots&\vdots&\ddots
 \end{pmatrix}\begin{array}{c}
                \mathcal V_0 \\
                \mathcal V_1 \\
                \mathcal V_2 \\
                \vdots
              \end{array}.
\]

Recall that
\[
\mathcal V_n
 =\operatorname{span}\{f_n,Tf_n,T^*f_n,T^*Tf_n\}.
\]
We have
\begin{align*}
   \widetilde Tf_n
 =&(I+N_n)Tf_n
 =Tf_n-\lambda_nf_n=z_n\in \mathcal V_n,  \\
  T^*z_n= & T^*Tf_n-\lambda_nT^*f_n\in\mathcal V_n.
\end{align*}
Note that
\[
 T^*(I+N^*)f_n
 =T^*f_n-\frac{\overline{\lambda_n}}{\|z_n\|^2}T^*z_n\in\mathcal V_n.
\]
Since $I-N^*$ preserves $\mathcal V_n$, it follows that
\[
 \widetilde T^*f_n
 =(I-N^*)T^*(I+N^*)f_n\in\mathcal V_n.
\]
We also have
\[
 \langle\widetilde T^*f_n,f_n\rangle =\langle f_n,\widetilde Tf_n\rangle
 ={\langle f_n,z_n\rangle}=0.
\]

For every $n\ge1$, set
\[
 \mathcal H_{1,n}
 =\operatorname{span}\{\widetilde Tf_n,\widetilde T^*f_n\}
 , \qquad
 \mathcal H_0=\big(\bigoplus_{n\ge1}\mathcal H_{1,n}\big)^\perp.
\]
Then
\[
1\le\dim\mathcal H_{1,n}\le2,\qquad \mathcal H_{1,n}\subset\mathcal M^\perp \cap\mathcal V_n\cap \{f_n\}^\perp,
\]
 and
$(f_n)$ is an orthonormal sequence in
$\mathcal H_0\cap\mathcal M^\perp$.

Relative to $\mathcal H_0\oplus\mathcal H_{1,1}
 \oplus\mathcal H_{1,2}\oplus\cdots$, the operator $\widetilde T$ has the block form
\[
 \widetilde T=
 \begin{pmatrix}
 D&E_1&E_2&\cdots\\
 F_1&G_1&0&\cdots\\
 F_2&0&G_2&\cdots\\
 \vdots&\vdots&\vdots&\ddots
 \end{pmatrix} \begin{matrix}
   \mathcal H_{0} \\
   \mathcal H_{1,1} \\
   \mathcal H_{1,2} \\
   \vdots
 \end{matrix}.
\]
Moreover
\[
 Df_n=P_{\mathcal H_0}\widetilde Tf_n=0,
 \qquad
 D^*f_n=P_{\mathcal H_0}\widetilde T^*f_n=0.
\]

Since $S_0,S_0^{-1}$
and their adjoints preserve $\mathcal V_n$,
\[
 \max\left\{
 \|\widetilde T|_{\mathcal V_n}\|,
 \|\widetilde T^*|_{\mathcal V_n}\|
 \right\}
 \le
 \|S_0\|\|S_0^{-1}\|\varepsilon_n
 <2\varepsilon_n.
\]
Here, we use that $(1+\theta)^2<2$.
Consequently,
\[
 \begin{aligned}
 \|E_n\|,\ \|G_n\|
 &\le\|\widetilde T|_{\mathcal V_n}\|
 <2\varepsilon_n,\\
 \|F_n\|=\|F_n^*\|
 &\le\|\widetilde T^*|_{\mathcal V_n}\|
 <2\varepsilon_n.
 \end{aligned}
\]
\end{proof}
\begin{remark}\label{rem:tailchoice}
Fix a construction in Proposition~\ref{prop:form} and an integer $m\ge1$.
If $(\varepsilon_n')$ is positive and decreasing, with
\[
 \varepsilon_j'=\varepsilon_j\quad(j\le m),\qquad
 0<\varepsilon_n'\le\varepsilon_n\quad(n>m),
\]
then the proposition can be applied with the same $\theta,\mathcal M$
and the bounds $\varepsilon_n'$ so that
\[
 (f_j',\mathcal H_{1,j}',G_j')=(f_j,\mathcal H_{1,j},G_j)
 \qquad(1\le j\le m).
\]
Indeed, keep $x_j,y_j$ for $j\le m$ and repeat the induction for
$n>m$. Each later step has only finitely many orthogonality conditions,
and the compactness argument permits the smaller bound.

For $j\le m$, the vectors $f_j,z_j$, the space $\mathcal V_j$, and
$N_j$ remain fixed. The formulas
\[
 \begin{aligned}
 \widetilde Tf_j&=(I+N_j)Tf_j,\\
 \widetilde T^*f_j&=(I-N_j^*)T^*(I+N_j^*)f_j,\\
 P_{\mathcal V_j}\widetilde T|_{\mathcal V_j}
 &=(I+N_j)\bigl(P_{\mathcal V_j}T|_{\mathcal V_j}\bigr)(I-N_j)
 \end{aligned}
\]
therefore leave $\mathcal H_{1,j}$ and $G_j$ unchanged. All bounds
in the proposition hold with $\varepsilon_n'$.
\end{remark}

We combine Proposition~\ref{prop:form} with a rank-one
trace adjustment, as in \cite[Lemma 3.2]{Liu}.

\begin{lemma}\label{lem:shear}
Let $T$ be a compact operator of infinite rank, let $\gamma>0$,
and let $(\varepsilon_n)$ be a positive sequence decreasing to zero. There are an invertible operator
$S$, a unit vector $v$, an orthogonal decomposition
\[
 \mathcal H=\mathcal H_0\oplus\mathcal H_1,
 \qquad \mathcal H_1=\mathbb Cv\oplus\bigoplus_{n\ge1}\mathcal H_{1,n},
 \qquad  1\le\dim\mathcal H_{1,n}\le2,
\]
and an orthonormal sequence $(f_n)\subset\mathcal H_0$ such that
$S^{-1}TS$ has the form
\[
S^{-1}TS=
 \begin{pmatrix}
 \widehat D&E_v&E_1&E_2&\cdots\\
 F_v&\lambda&0&0&\cdots\\
 F_1&0&G_1&0&\cdots\\
 F_2&0&0&G_2&\cdots\\
 \vdots&\vdots&\vdots&\vdots&\ddots
 \end{pmatrix}
 \begin{matrix}
 \mathcal H_0\\\mathbb Cv\\\mathcal H_{1,1}\\\mathcal H_{1,2}\\\vdots
 \end{matrix},
\]
where $\lambda=-\sum_{n=1}^{\infty}\tr G_n$ and $\sum_{n=1}^{\infty}\|G_n\|_1<\infty$.
The blocks satisfy
\[
 \|G_n\|_1\le4\varepsilon_n,\qquad
 \|E_n\|\le2\varepsilon_n,\qquad \|F_n\|\le2\varepsilon_n,
\]
\[
\|S\|\|S^{-1}\|<1+\gamma,
 \qquad \|E_v\|<\gamma\|T\|,
 \qquad \|F_v\|<\gamma\|T\|,
\]
and
\[
(S^{-1}TS)f_n,\ (S^{-1}TS)^*f_n\in\mathcal H_{1,n}.
\]
In particular, $\widehat D$ has an infinite-dimensional reducing null subspace.
\end{lemma}

\begin{proof}
Choose
\[
\theta=\min\left\{
 \frac{1}{8},\frac{\gamma}{16}\right\},\qquad
 \delta=\frac{\theta}{6},
 \quad
 a_n=\min\{\varepsilon_n, \frac{\delta^2}{2^{n+2}}\|T\|\}.
\]

Since $T^*$ is nonzero and compact,
there is a unit vector $x$ with $\|T^*x\|=\|T\|$.
Choose a unit vector
$
 y\in
 \bigl(\operatorname{span}\{x,Tx,T^*x\}\bigr)^\perp
$ such that
\[
 \max\{\|Ty\|,\|T^*y\|\}\le\delta^2\|T\|.
\]

Let
\[
 v=\frac{\delta x+y}{\sqrt{1+\delta^2}},
 \qquad
\lambda_0=\langle Tv,v\rangle,
 \qquad
 z= T^*v-\overline{\lambda_0} v.
\]
Apply Lemma~\ref{lem:vector-estimate} to $T^*$.
Since $\|T^*x\|=\|T^*\|=\|T\|$ and
$\langle T^*v,v\rangle=\overline{\lambda_0}$, it gives
\[
 |\lambda_0|\le2\delta^2\|T\|,
 \qquad
 2\delta\|T\|\geq\|z\|\ge\frac{\delta}{2}\|T\|,
\qquad
\|Tv\|,\|T^*v\|<2\delta\|T\|.
\]
Set
$
 u:=\frac{z}{\|z\|}.
$
Then $u$ is a unit vector with  $ \langle u,v\rangle=0$, and
$
 \langle Tu,v\rangle=\|z\|.
$

Apply Proposition~\ref{prop:form} with the similarity parameter
$\theta$,  the sequence $a_n$, and
\[
 \mathcal M=\operatorname{span}\{u,v,Tu,Tv\}.
\]
Let $S_0$, $(f_n)$, and $\mathcal H_{1,n}$ be the
resulting operator, vectors, and spaces, respectively. Set
\[
 \mathcal H_1=\mathbb Cv\oplus
             \bigoplus_{n\ge1}\mathcal H_{1,n},
 \qquad
 \mathcal H_0=\mathcal H_1^\perp,\qquad\widetilde T=S_0^{-1}TS_0.
\]
Since $S_0$, $S_0^*$, and their inverses act as the identity on
$\mathcal M$, we have $\widetilde Tu=S_0^{-1}Tu=Tu$ and
\[
\begin{aligned}
\widetilde Tv
 &=S_0^{-1}TS_0v
 =S_0^{-1}Tv=Tv\in \mathcal M,\\
\widetilde T^*v
 &=S_0^*T^*(S_0^{-1})^*v
 =S_0^*T^*v=T^*v=z+\overline{\lambda_0} v\in\mathcal M.
\end{aligned}
\]
Therefore,
\[
P_{\mathcal H_{1,n}}\widetilde Tv
=P_{\mathcal H_{1,n}}Tv=0.
\]
Conversely, for every $w\in\mathcal H_{1,n}$,
\[
\langle\widetilde Tw,v\rangle
=\langle w,\widetilde T^*v\rangle=0,\qquad\langle\widetilde T^*w,v\rangle
=\langle w,\widetilde Tv\rangle=0,
\]
 Thus
\[
P_{\mathbb Cv}\widetilde T|_{\mathcal H_{1,n}}=0.
\]

Let
\[
 D_0=P_{\mathcal H_0}\widetilde T|_{\mathcal H_0},
 \qquad E_{v,0}=P_{\mathcal H_0}\widetilde T|_{\mathbb Cv},
 \qquad F_v=P_{\mathbb Cv}\widetilde T|_{\mathcal H_0}.
\]
\[
E_n=P_{\mathcal H_0}\widetilde T|_{\mathcal H_{1,n}},\qquad
F_n=P_{\mathcal H_{1,n}}\widetilde T|_{\mathcal H_0},\qquad
G_n=P_{\mathcal H_{1,n}}\widetilde T|_{\mathcal H_{1,n}}.
\]
By Proposition~\ref{prop:form}, $ \widetilde T$ has the form
\[
 \widetilde T=
 \begin{pmatrix}
 D_0&E_{v,0}&E_1&E_2&\cdots\\
 F_v&\lambda_0&0&0&\cdots\\
 F_1&0&G_1&0&\cdots\\
 F_2&0&0&G_2&\cdots\\
 \vdots&\vdots&\vdots&\vdots&\ddots
 \end{pmatrix}
 \begin{matrix}
 \mathcal H_0\\\mathbb Cv\\\mathcal H_{1,1}\\\mathcal H_{1,2}\\\vdots
 \end{matrix}.
\]
The estimates in Proposition~\ref{prop:form} give
\[
 \begin{aligned}
 &  \max\{\|S_0\|,\|S_0^{-1}\|\}<1+\theta,\\
 \|E_n\|&\le\|\widetilde T|_{\mathcal H_{1,n}}\|
              \le(1+\theta)^2a_n<2a_n,\\
 \|F_n\|&=\|F_n^*\|
          \le\|\widetilde T^*|_{\mathcal H_{1,n}}\|
              \le(1+\theta)^2a_n<2a_n.
 \end{aligned}
\]
The trace norm satisfies
\[
\sum_{n=1}^{\infty}|\operatorname{Tr}G_n|
\le\sum_{n=1}^{\infty}\|G_n\|_1
\le4\sum_{n=1}^{\infty}a_n
\le\delta^2\|T\|<\infty.
\]
Here $\|G_n\|_1\le(\dim\mathcal H_{1,n})\|G_n\|\le4a_n$.

Now we change only the first row and the $\mathbb Cv$ diagonal entry
so that the complementary compression has trace zero.

Since $u\in\mathcal M$, $\langle u,v\rangle=0$, and
$\mathcal H_{1,n}\subseteq\mathcal M^\perp$, we have $u\in\mathcal H_0$.
Using $Tu\in\mathcal M$ and
$\widetilde T^*v=\overline{\lambda_0}v+\|z\|u$, we obtain
\begin{align*}
  F_vu= &\,P_{\mathbb Cv}\widetilde Tu
=\langle Tu,v\rangle v
=\langle u,T^*v\rangle v
=\|z\|v. \\
  F_v^*v=&\,P_{\mathcal H_0}\widetilde T^*v
=P_{\mathcal H_0}
  \bigl(\overline{\lambda_0}v+\|z\|u\bigr)
=\|z\|u.
\\
   F_nu=&\,P_{\mathcal H_{1,n}}\widetilde T u=P_{\mathcal H_{1,n}} T u=0, \quad n\ge1.
\end{align*}
Now  let $\tau=\lambda_0+\sum_{n=1}^{\infty}\tr G_n$ and define the operators
\[
R=-\frac{\tau}{\|z\|}\,u\otimes v^*:
                \mathbb Cv\longrightarrow\mathcal H_0
\]
and
\[
 S_1=
 \begin{pmatrix}
 I & R&0&0&\cdots\\
 0&1&0&0&\cdots\\
 0&0&I&0&\cdots\\
 0&0&0&I&\cdots\\
 \vdots&\vdots&\vdots&\vdots&\ddots
 \end{pmatrix}
 \begin{matrix}
 \mathcal H_0\\\mathbb Cv\\\mathcal H_{1,1}\\\mathcal H_{1,2}\\\vdots
 \end{matrix}.
\]
Its inverse $S_1^{-1}$ is obtained by replacing $R$ by $-R$.

For every \(w\in\mathcal H_0\), since \(F_vw\in\mathbb Cv\) and \(v\) is a unit vector,
\[
RF_vw
=\langle w,F_v^*v\rangle Rv
=\|z\|\langle w,u\rangle
  \left(-\frac{\tau}{\|z\|}u\right)=-\tau P_uw.
\]
\[
F_vRv
=F_v(-\frac{\tau}{\|z\|}u)
=-\tau v.
\]
Consequently,
\[
 RF_v=-\tau P_u,\qquad F_vR=-\tau I_{\mathbb Cv}.
\]
Now we calculate the first $2\times2$ corner:
\[
 \begin{pmatrix}I&-R\\0&1\end{pmatrix}
 \begin{pmatrix}D_0&E_{v,0}\\F_v&\lambda_0\end{pmatrix}
 \begin{pmatrix}I&R\\0&1\end{pmatrix}
 =\begin{pmatrix}\widehat D&E_v\\F_v&\lambda\end{pmatrix}
 \begin{matrix}\mathcal H_0\\\mathbb Cv\end{matrix},
\]
where
\[
 \begin{aligned}
 \widehat D&=D_0+\tau P_u,\\
 E_v&=E_{v,0}+(D_0+(\tau-\lambda_0)I)R,\\
 \lambda&=\lambda_0-\tau=-\sum_{n=1}^{\infty}\tr G_n.
 \end{aligned}
\]
Block multiplication leaves $E_n,F_n,G_n$ unchanged. The block
from $\mathbb Cv$ to $\mathcal H_{1,n}$ becomes $F_nR$, which is zero
because $\operatorname{ran}R\subseteq\mathbb Cu$ and $F_nu=0$.
Thus $S=S_0S_1$ gives
\[
 S^{-1}TS= S_1^{-1}\widetilde{T}S_1=
 \begin{pmatrix}
 \widehat D&E_v&E_1&E_2&\cdots\\
 F_v&\lambda&0&0&\cdots\\
 F_1&0&G_1&0&\cdots\\
 F_2&0&0&G_2&\cdots\\
 \vdots&\vdots&\vdots&\vdots&\ddots
 \end{pmatrix}
 \begin{matrix}
 \mathcal H_0\\\mathbb Cv\\\mathcal H_{1,1}\\\mathcal H_{1,2}\\\vdots
 \end{matrix}.
\]
Moreover, $f_n$ and $\mathcal H_{1,n}$ lie in
$\mathcal M^\perp$. Thus $S_1$, $S_1^{-1}$, and their
adjoints fix these vectors, and
\[
 \begin{aligned}
 (S^{-1}TS)f_n
 &=\widetilde Tf_n\in\mathcal H_{1,n},\\
 (S^{-1}TS)^*f_n
 &=\widetilde T^*f_n\in\mathcal H_{1,n}.
 \end{aligned}
\]
Since $\mathcal H_{1,n}\subseteq\mathcal H_0^\perp$, compression gives
\[
 \widehat Df_n=\widehat D^*f_n=0\quad(n\ge1).
\]
Thus $\overline{\operatorname{span}}\{f_n:n\ge1\}$ is an
infinite-dimensional reducing null subspace of $\widehat D$.

For the remaining estimate, note that
\[
 |\tau|
 \le |\lambda_0|+\sum_{n=1}^{\infty}\|G_n\|_1
 \le2\delta^2\|T\|+4\sum_{n=1}^{\infty}a_n
 \le3\delta^2\|T\|.
\]
Consequently,
\[
 \|R\|=\frac{|\tau|}{\|z\|}<6\delta
 =\theta,
 \qquad
 \max\{\|S_1\|,\|S_1^{-1}\|\}\le1+\theta.
\]
Together with
$\max\{\|S_0\|,\|S_0^{-1}\|\}<1+\theta$, this gives
\[
 \|S\|\|S^{-1}\|
 <(1+\theta)^4
 \le1+15\theta
 \le1+\gamma.
\]

Using \(F_v^*v=\|z\|u\) and \(\|u\|=1\), we obtain
\[
 \|F_v\|= \|F_v^*\|=\|z\|<2\delta\|T\|\le\gamma\|T\|.
\]
For $E_v$, using $\|E_{v,0}\|\le\|Tv\|<2\delta\|T\|$,  $|\tau-\lambda_0|\leq\delta^2\|T\|$ and
\[
\|D_0\| \le (1+\theta)^2\|T\| <2\|T\|.
\]
Then we obtain
\[
\|E_v\|\le\|E_{v,0}\|+\bigl(\|D_0\|+|\tau-\lambda_0|\bigr)\|R\|
 <(3\delta+2\theta)\|T\|\le\gamma\|T\|.
\]
\end{proof}
\begin{remark}\label{rem:sheartail}
Fix a construction in Lemma~\ref{lem:shear} and an integer
$m\ge1$. The bounds on the tail can be made smaller while
preserving $v$ and the first $m$ local constructions.

More precisely, let $(\varepsilon_n')$ be a positive decreasing
sequence satisfying
\[
 \varepsilon_j'=\varepsilon_j\quad(1\le j\le m),
 \qquad
 0<\varepsilon_n'\le\varepsilon_n\quad(n>m).
\]
Then the construction can be repeated with the same $\gamma$
and the new bounds $(\varepsilon_n')$ so that
\[
 v'=v,\qquad
 f_j'=f_j,\qquad
 \mathcal H_{1,j}'=\mathcal H_{1,j},\qquad
 G_j'=G_j
 \quad(1\le j\le m).
\]

Indeed, retain $x,y,\theta,\delta$ from the proof of
Lemma~\ref{lem:shear} and apply Remark~\ref{rem:tailchoice} with
\[
 a_n'=\min\left\{
 \varepsilon_n',
 \frac{\delta^2}{2^{n+2}}\|T\|
 \right\}.
\]
Since $a_j'=a_j$ for $j\le m$, the first $m$ local
constructions can be kept unchanged. Construct the remaining
blocks using the smaller bounds $a_n'$, and then repeat the
rank-one trace adjustment. All estimates remain valid because
$a_n'\le a_n$.

The scalar entry in the new construction is
\[
 \lambda'
 =-\sum_{j=1}^{m}\tr G_j-\sum_{j>m}\tr G_j'.
\]
Thus the changes to the tail may change its total trace and
hence $\lambda'$. The similarity $S'$ may also change.
The data preserved by this construction are precisely
$v$ and the triples $(f_j,\mathcal H_{1,j},G_j)$ for
$1\le j\le m$.
\end{remark}

\section{Finite-dimensional approximation and block coupling}
\label{sec:glue}

We construct commutators by finite-dimensional approximation while
controlling the diagonalization of one factor. Resolvent estimates
then allow us to solve the two off-diagonal equations.
Lemma~\ref{lem:approximation} combines these constructions into a
block commutator representation with the required norm estimates.

\begin{definition}\label{def:spectral-gap}
Let $A\in M_m(\mathbb C)$ ($m\geq 2$) be a normal matrix, with eigenvalues
$\mu_1,\ldots,\mu_m$ counted with multiplicity. Define its spectral gap by
\[
 \operatorname{gap}(A)
 =\min_{i\ne j}|\mu_i-\mu_j|.
\]
\end{definition}
We first estimate how a small perturbation affects the spectral gap of a normal matrix.
\begin{proposition}\label{prop:spectral-gap}
Let $A,B\in M_m(\mathbb{C})$ be normal, with $m\geqslant2$, and suppose
\begin{equation*}
\lVert A-B\rVert <\frac{1}{2}\operatorname{gap}(A).
\end{equation*}
Then
\begin{equation*}
\operatorname{gap}(B)\geqslant\operatorname{gap}(A)-2\lVert A-B\rVert.
\end{equation*}
\end{proposition}
\begin{proof}
Let $g=\operatorname{gap}(A)$ and $\delta=\lVert A-B\rVert <g/2$.
Let $\lambda_1,\ldots,\lambda_m$ be the eigenvalues of $A$, and choose unit
vectors $x_i$ with $Ax_i=\lambda_i x_i$. Since $B$ is normal, choose an
orthonormal basis $\{e_j\}_{j=1}^{m}$ with $Be_j=\nu_j e_j$. Write
\begin{equation*}
x_i=\sum_{j=1}^{m}c_{ij}e_j,\qquad\sum_{j=1}^{m}|c_{ij}|^2=1.
\end{equation*}
Set $d_i=\operatorname{dist}\bigl(\sigma(B),\lambda_i\bigr)$.
Since $d_i\leqslant|\nu_j-\lambda_i|$ for every $j$,
\begin{equation*}
\begin{aligned}
d_i^2=d_i^2\sum_{j=1}^{m}|c_{ij}|^2
&\leqslant\sum_{j=1}^{m}|\nu_j-\lambda_i|^2|c_{ij}|^2\\
&=\left\lVert\sum_{j=1}^{m}c_{ij}(\nu_j-\lambda_i)e_j\right\rVert^2\\
&=\lVert (B-A)x_i\rVert^2\leq\delta^2.
\end{aligned}
\end{equation*}
Choose $\mu_i\in\sigma(B)$ with $|\lambda_i-\mu_i|\leqslant\delta$.
For $i\ne j$,
\begin{equation*}
|\mu_i-\mu_j|\geqslant|\lambda_i-\lambda_j|-|\mu_i-\lambda_i|-|\mu_j-\lambda_j|
\geqslant g-2\delta>0.
\end{equation*}
Thus $\mu_1,\ldots,\mu_m$ are distinct and list all eigenvalues of $B$.
Taking the minimum over $i\ne j$ proves the assertion.
\end{proof}

We next control the diagonalizing similarity for a small perturbation of a diagonal matrix.

\begin{lemma}\label{lem:diagonalization}
Let $\Lambda=\operatorname{diag}(\nu_1,\ldots,\nu_m)$, $m\geqslant2$, and
suppose that $\operatorname{gap}(\Lambda)\geqslant g>0$.
Let $0<\varepsilon\leqslant1/4$. If
$
\lVert E\rVert\leqslant g\varepsilon/8,
$
then there exist an invertible $V\in M_m(\mathbb{C})$ and a diagonal matrix
$\widetilde{\Lambda}$ such that
\begin{equation*}
V^{-1}(\Lambda+E)V=\widetilde{\Lambda},\qquad
\lVert V-I\rVert\leqslant\varepsilon,\qquad
\lVert\widetilde{\Lambda}-\Lambda\rVert\leqslant\lVert E\rVert.
\end{equation*}
\end{lemma}

\begin{proof}
For $0\le t\le1$, if
$\operatorname{dist}(z,\sigma(\Lambda))>\|E\|$, then
\[
 \|t(zI-\Lambda)^{-1}E\|
 \le\frac{\|E\|}{\operatorname{dist}(z,\sigma(\Lambda))}
 <1.
\]
Thus $zI-\Lambda-tE$ is invertible, and
\[
 \sigma(\Lambda+tE)
 \subseteq
 \bigcup_{j=1}^m
 \{z\in\mathbb C:|z-\nu_j|\le\|E\|\}.
\]
These $m$ disks are mutually disjoint because $2\|E\|<g$.
The eigenvalues depend continuously on $t$ and remain in these fixed disks. By a standard continuity argument, the eigenvalue count in each disk is independent of $t$; hence, as at $t=0$, there is exactly one eigenvalue per disk for all $t \in [0,1]$.
In particular, $\Lambda+E$ has distinct eigenvalues
$\widetilde\nu_1,\ldots,\widetilde\nu_m$ satisfying
\[
 |\widetilde\nu_j-\nu_j|\le\|E\|,
 \qquad 1\le j\le m.
\]

Let $e_1,\ldots,e_m$ be the standard orthonormal basis
of $\mathbb C^m$, so that $\Lambda e_j=\nu_j e_j$.
For each $j$, choose a nonzero eigenvector $f_j$ of
$\Lambda+E$ corresponding to $\widetilde\nu_j$.
We first show that $f_j\notin e_j^\perp$.
Indeed, since
\[
|\nu_i-\widetilde\nu_j|
\ge |\nu_i-\nu_j|-|\nu_j-\widetilde\nu_j|
\ge g-\|E\|
\qquad(i\ne j),
\]
the assumption $f_j\in e_j^\perp$ would give
\[
(g-\|E\|)\|f_j\|
\le \|(\Lambda-\widetilde\nu_j I)f_j\|
= \|Ef_j\|
\le \|E\|\|f_j\|,
\]
which contradicts $\|E\|<g/2$.

Divide $f_j$ by its $j$th coordinate, write it as $\widetilde{f}_j$, so that  the $j$th coordinate of $\widetilde{f}_j$ is $1$.  Since eigenvectors corresponding to distinct eigenvalues
are linearly independent, the vectors
$\widetilde f_1,\ldots,\widetilde f_m$ form a basis of
$\mathbb C^m$.  Define
\[
Ve_j=\widetilde{f}_j,
\qquad
\widetilde\Lambda=
\begin{pmatrix}
\widetilde\nu_1&&0\\
&\ddots&\\
0&&\widetilde\nu_m
\end{pmatrix}.
\]
Then $V$
is invertible and all the diagonal entries are 1. Moreover,
\[
(\Lambda+E)V=V\widetilde\Lambda,
\qquad
\|\widetilde\Lambda-\Lambda\|\le\|E\|.
\]

Now  we estimate $\|V-I\|$. Let \(P\) be a coordinate projection.  Then $P$ commutes with both $\Lambda$ and
$\widetilde\Lambda$. Suppose first that $P\ne0,I$. Multiplying
$\Lambda V-V\widetilde\Lambda=-EV$ on the left by $P$
and on the right by $I-P$ gives
\[
\bigl(\Lambda|_{P\mathbb C^m}\bigr)PV(I-P)
-
PV(I-P)
\bigl(\widetilde\Lambda|_{(I-P)\mathbb C^m}\bigr)
=-PEV(I-P).
\]
If $Pe_i=e_i$ and
$Pe_j=0$, then $i\ne j$, and hence
$|\nu_i-\widetilde\nu_j|\ge g-\|E\|.$ Considering these operators as restricted to the space, $\sigma(\Lambda|_{P\mathbb C^m})$ and $\sigma(\widetilde\Lambda|_{(I-P)\mathbb C^m})$ are separated by at least $g-\|E\|$. By the estimate for normal operators in
\cite[Section~10, p.~15]{BR1997}, we obtain
\[
\|PV(I-P)\|
\le \frac{3}{g-\|E\|}\|PEV(I-P)\|\le\frac{3\|E\|}{g-\|E\|}\|V\|.
\]
Replacing $P$ by $I-P$ gives the same bound for
$\|(I-P)VP\|$. Relative to $P\mathbb C^m\oplus(I-P)\mathbb C^m$,
these operators occupy opposite off-diagonal blocks.
Consequently,
\[
\begin{aligned}
 \|PV(I-P)+(I-P)VP\|
 &=\max\{\|PV(I-P)\|,\|(I-P)VP\|\}\\
 &\le\frac{3\|E\|}{g-\|E\|}\|V\|.
\end{aligned}
\]
This also holds for $P=0,I$.

Since  $V-I$ has zero diagonal, using the identity
\[
V-I
=\frac{2}{2^m}
\sum_P\bigl(PV(I-P)+(I-P)VP\bigr),
\]
where the sum on the right-hand side ranges over all standard projections.

Then we have
\[
\|V-I\|
\le \frac{6\|E\|}{g-\|E\|}\|V\|
\le \frac{6\|E\|}{g-\|E\|}
       \bigl(1+\|V-I\|\bigr).
\]
Since $\|E\|\le\varepsilon g/8$ and
$0<\varepsilon\le1/4$, rearranging gives
\[
\|V-I\|
\le \frac{6\|E\|}{g-7\|E\|}
\le \frac{6\varepsilon}{8-7\varepsilon}
\le \varepsilon.
\]
\end{proof}

Combining Lemmas~\ref{lem:extension}, \ref{lem:correction},
and~\ref{lem:diagonalization}, we enlarge a regular pair by one
dimension while controlling the factors and their diagonalization.
\begin{lemma}\label{lem:onestep}
Let $ 0\neq D=[A,B] \in M_m(\mathbb C)$, $p>2$, $0<\gamma\le1/8$. Suppose that $(A,B)$ is  a regular pair and
$\|\mathcal R_{A,B}\|\le L$. Suppose
$W^{-1}AW=\Lambda$ is diagonal with distinct entries and
$\|W-I_m\|\le1/4$. Fix $|\lambda|\le1$ outside $\sigma(A)$.
There exists $\delta>0$ such that  every trace-zero matrix $H\in M_{m+1}(\mathbb C)$ with $\|H\|_p\le\delta$ admits a regular
pair $(\widetilde A,\widetilde B)$ and an invertible $\widetilde W$
satisfying
\[
 [\widetilde A,\widetilde B]=\begin{pmatrix}D&0\\0&0\end{pmatrix}+H,
 \qquad \|\mathcal R_{\widetilde A,\widetilde B}\|\le \frac{1}{4\sqrt{\delta}},
 \qquad \widetilde W^{-1}\widetilde A\widetilde W=\widetilde\Lambda,
\]
where $\widetilde\Lambda$ is a diagonal matrix and
\begin{equation*}
 \begin{gathered}
 \left\|\widetilde A-\begin{pmatrix}A&0\\0&\lambda\end{pmatrix}\right\|_p,\quad
 \left\|\widetilde B-\begin{pmatrix}B&0\\0&0\end{pmatrix}\right\|_p\le\gamma,\\
 \left\|\widetilde W-\begin{pmatrix}W&0\\0&1\end{pmatrix}\right\|,\quad
 \left\|\widetilde\Lambda-\begin{pmatrix}\Lambda&0\\0&\lambda\end{pmatrix}\right\|\le\gamma.
 \end{gathered}
\end{equation*}
\end{lemma}
\begin{proof}
Choose $K\ge1$ and $0<g\le1$ such that
$\|A\|_p,\|B\|_p\le K$, and   $\operatorname{gap}
\begin{pmatrix}\Lambda&0\\0&\lambda\end{pmatrix}
\ge g$. Set
\[
 R=(A-\lambda I)^{-1}
   =W(\Lambda-\lambda I)^{-1}W^{-1},
\]
The Neumann series and $\|W-I\|\le1/4$ give
\[
\|R\|\leq \frac{5}{4}\cdot\frac{4}{3}\cdot\frac{1}{g}<\frac{2}{g}.
\]
Choose a unit vector $x$ with $\|Dx\|=\|D\|$.  Using \(\|A\|\le\|A\|_p\le K\) and \(|\lambda|\le1\), we obtain
\[
\|D\|=\|Dx\|
\le \|A-\lambda I\|\,\|RDx\|
\le(K+1)\|RDx\|
\]
and
\[
\|RDx\|\le \|R\|\|[A,B]\|\leq\frac{4K^2}{g}.
\]

Set
\[
 \theta=\frac{\gamma g^2}{64K^2}.
\]
Then $0<\theta<1/8$, and the preceding estimates give
\[
 \theta\|R\|\|RDx\|
 \le\frac{1}{8}\gamma<\frac18.
\]
Applying Lemma~\ref{lem:extension} gives the regular pair
\[
 \widehat A=
 \begin{pmatrix}
 A&\theta(A-\lambda I)x\\
 0&\lambda
 \end{pmatrix},
 \qquad
 \widehat B=
 \begin{pmatrix}
 B&\theta RB(A-\lambda I)x\\
 0&0
 \end{pmatrix},
\]
with
\[
 [\widehat A,\widehat B]
 =\begin{pmatrix}D&0\\0&0\end{pmatrix},
 \qquad
 \|\mathcal R_{\widehat A,\widehat B}\|
 \le\widehat L,
\]
where
\[
 \widehat L
 :=5L+\frac{(3g+4K)(K+1)}{g\theta\|D\|}.
\]
Since each column vector has rank at most one, its
Schatten $p$-norm equals its operator norm. Hence
\[
 \left\|\widehat A-
 \begin{pmatrix}A&0\\0&\lambda\end{pmatrix}\right\|_p
\le\theta(K+1)\le\frac{\gamma g}{32},
\]
\[
 \left\|\widehat B-
 \begin{pmatrix}B&0\\0&0\end{pmatrix}\right\|_p
\le\theta\|R\|\|B\|\|A-\lambda I\|
\le\frac{\gamma g(K+1)}{32K}
\le\frac{\gamma g}{16}.
\]

Take
\[
\delta=\min\left\{\frac{1}{64\widehat L^2},
                 \frac{\gamma g}{128\widehat L}\right\}.
\]
By Lemma~\ref{lem:correction}, for every trace-zero $H$ with $\|H\|_p\le\delta$, there are $X,Y$ with
\[
 \|X\|_p,\ \|Y\|_p\le  2\widehat L \|H\|_p  \leq   \frac{\gamma g}{64}
\]
such that
\[
[\widehat A+X,\widehat B+Y]=[\widehat A,\widehat B]+H.
\]
Set
$
 \widetilde A=\widehat A+X,\,
 \widetilde B=\widehat B+Y.
$
Then $(\widetilde A,\widetilde B)$ is a regular pair and $\|\mathcal R_{\widetilde A,\widetilde B}\|\le 2\widehat L.$
Moreover,
\[
\left\|\widetilde A-
\begin{pmatrix}A&0\\0&\lambda\end{pmatrix}\right\|_p
\le\frac{3\gamma g}{64},
\qquad
\left\|\widetilde B-
\begin{pmatrix}B&0\\0&0\end{pmatrix}\right\|_p
\le\frac{3\gamma g}{64}<\gamma.
\]

Now we set
\[
 U=\begin{pmatrix}W&0\\0&1\end{pmatrix}.
\]
Since $\|W-I\|\le1/4$,
\[
 \|U\|\le\frac54,\qquad
 \|U^{-1}\|\le\frac43,\qquad
 \|U\|\|U^{-1}\|<2.
\]
It follows that
\[
\begin{aligned}
 \left\|U^{-1}\widetilde AU-
 \begin{pmatrix}\Lambda&0\\0&\lambda\end{pmatrix}\right\|
 &=
 \left\|U^{-1}\left(
 \widetilde A-
 \begin{pmatrix}A&0\\0&\lambda\end{pmatrix}
 \right)U\right\|\\
 &\le2\cdot\frac{3\gamma g}{64}
 =\frac{3\gamma}{4}\cdot\frac{g}{8}.
\end{aligned}
\]
Apply Lemma~\ref{lem:diagonalization} for
$3\gamma/4$ (in place of $\varepsilon$), $g$ and the above matrix, there exist an invertible matrix $V$ and
a diagonal matrix $\widetilde\Lambda$ such that
\[
 V^{-1}U^{-1}\widetilde AUV=\widetilde\Lambda,
 \qquad
 \|V-I\|\le\frac{3\gamma}{4},
\]
and
\[
 \left\|\widetilde\Lambda-
 \begin{pmatrix}\Lambda&0\\0&\lambda\end{pmatrix}\right\|
 \le\frac{3\gamma g}{32}<\gamma.
\]
Define $\widetilde W=UV$. Then
\[
 \left\|\widetilde W-
 \begin{pmatrix}W&0\\0&1\end{pmatrix}\right\|
 =\|U(V-I)\|
 \le\frac{15\gamma}{16}<\gamma.
\]

The choice of $\theta$ uses only $K,g,\gamma$.
If $\|D\|\ge d>0$ and we replace $\|D\|$ by $d$, then the choice of
$\delta$ depends only on $L, K,g,\gamma,d$.
\end{proof}
To apply this construction to successive compressions, we correct
their traces and estimate the resulting increments.
\begin{lemma}\label{lem:trace-truncations}
Let $G\in \mathcal{S}_1(\mathcal H)$ with $\tr G=0$, let $(P_n)_{n\ge0}$ be increasing
finite-rank orthogonal projections, and let $e$ be a unit vector in
$\operatorname{ran}P_0$. Set
$G^{(n)}=P_nGP_n-\tr(P_nGP_n)P_e$. Then
\[
\begin{aligned}
 \|G-G^{(n)}\|_1&\le2\|G-P_nGP_n\|_1,\\
 \|G^{(n)}-G^{(m)}\|_1&\le2\|G-P_mGP_m\|_1,
 \quad 0\le m\le n.
 \end{aligned}
\]
All finite-rank operators are viewed on the original space.
\end{lemma}
\begin{proof}
Since $\tr G=0$,
\[
 |\tr(P_nGP_n)|=|\tr(G-P_nGP_n)|\le\|G-P_nGP_n\|_1,
\]
which proves the first estimate. For $m\le n$, set
$H=P_n(G-P_mGP_m)P_n$. Then
\[
 G^{(n)}-G^{(m)}=H-(\tr H)P_e,
\]\[
 \|H-(\tr H)P_e\|_1\le2\|H\|_1
 \le2\|G-P_mGP_m\|_1.
\]
\end{proof}
We next use the resolvent estimates from Lemma~\ref{lem:circle}
to solve the two off-diagonal equations with small trace norms.
\begin{lemma}\label{lem:coupling}
Let $0<\beta\le1/8$, $p>2$ and $M\ge1$, and set $c=3(2M)^p$
as in Lemma~\ref{lem:circle}. Suppose that $C\in \mathcal{S}_p(\mathcal H_0)$
satisfies $\|C\|\le1$ and $\|C\|_p\le M$, and let $(\lambda_n)$
be a resolvent sequence supplied by that lemma.
There exists $\eta=\eta(p,M,\beta)>0$ with the following property.

Suppose that $\mathcal H_1=\mathbb C^m\oplus\ell^2(\mathbb N)$,
where $1\le m\le3$, and that
\[
 W^{-1}AW=\Lambda=
 \begin{pmatrix}
 \mu_1^h&&&&0\\
 &\ddots&&&\\
 &&\mu_m^h&&\\
 &&&\mu_1&\\
 0&&&&\ddots
 \end{pmatrix}
 \begin{matrix}
 \mathbb Ce_1^h\\
 \vdots\\
 \mathbb Ce_m^h\\
 \mathbb Ce_1\\
 \vdots
 \end{matrix},
\]
where the right vectors form an orthonormal basis,
$W$ is invertible, and
\[
 |\mu_i^h|>\frac{11}{10}\quad(1\le i\le m),
 \qquad \|W\|,\ \|W^{-1}\|\le2.
\]
Assume further that
\[
 |\mu_n-\lambda_n|,\quad
 \|We_n-e_n\|,\quad
 \|(W^{-1})^*e_n-e_n\|
 \le e^{-(c+3)n^3},\quad n\ge1.
\]
If bounded operators $E:\mathcal H_1\to\mathcal H_0$ and
$F:\mathcal H_0\to\mathcal H_1$ satisfy
\begin{equation}\label{eq:EFhyp} \tag{$*$}
 \|E\|,\ \|F\|<\eta,\qquad
 \|Ee_n\|,\ \|F^*e_n\| \le\min\{\eta,e^{-n^4}\},\quad n\ge1,
\end{equation}
then there are $X\in \mathcal{S}_1(\mathcal H_1,\mathcal H_0)$ and
$Y\in \mathcal{S}_1(\mathcal H_0,\mathcal H_1)$ with
$\|X\|_1,\|Y\|_1<\beta/2$ such that

$$
CX-XA=E,\qquad  AY-YC=F.$$
\end{lemma}

\begin{proof}
For $1\le i\le m$, define
\[
 \xi_i^h=(C-\mu_i^hI)^{-1}EWe_i^h,\qquad
 \zeta_i^h=(C^*-\overline{\mu_i^h}I)^{-1}
       F^*(W^{-1})^*e_i^h.
\]
For $n\ge1$, define
\[
 \xi_n
 =(C-\mu_nI)^{-1}EWe_n,\qquad
 \zeta_n
 =(C^*-\overline{\mu_n}I)^{-1}
       F^*(W^{-1})^*e_n.
\]
All these vectors belong to $\mathcal H_0$.

Relative to the ordered basis
$e_1^h,\ldots,e_m^h,e_1,e_2,\ldots$ of $\mathcal H_1$, set
\[
 \widetilde X=
 \begin{pmatrix}
 \xi_1^h&\cdots&\xi_m^h&\xi_1&\xi_2&\cdots
 \end{pmatrix}
 \begin{matrix}\mathcal H_0\end{matrix},
\qquad
 \widetilde Y=-
 \begin{pmatrix}
 (\zeta_1^h)^*\\
 \vdots\\
 (\zeta_m^h)^*\\
 \zeta_1^*\\
 \zeta_2^*\\
 \vdots
 \end{pmatrix}
 \begin{matrix}
 \mathbb Ce_1^h\\
 \vdots\\
 \mathbb Ce_m^h\\
 \mathbb Ce_1\\
 \mathbb Ce_2\\
 \vdots
 \end{matrix}.
\]
Thus $\widetilde X:\mathcal H_1\to\mathcal H_0$ and
$\widetilde Y:\mathcal H_0\to\mathcal H_1$.

For $1\le i\le m$, using \(
\|(C-\mu_i^hI)^{-1}\|
<10
\) and
$\|W\|,\|W^{-1}\|\le2$, we obtain
\[
 \|\xi_i^h\|\leq 10\|E\|\|W\|< 20\eta,\qquad
 \|\zeta_i^h\|\leq 10\|F^*\|\|W^{-1}\|<20\eta.
\]
For $n\ge1$, use
\[
 \|Ee_n\|,\ \|F^*e_n\|
 \le\min\{\eta,e^{-n^4}\}.
\]
By Lemma~\ref{lem:circle} and the assumption,
\[
 |\mu_n-\lambda_n|\,\|(C-\lambda_nI)^{-1}\|
 \le e^{-3n^3}<\frac{1}{2}.
\]
Then the Neumann series shows that
\[
C-\mu_nI
=(C-\lambda_nI)
\bigl(I-(\mu_n-\lambda_n)(C-\lambda_nI)^{-1}\bigr)
\]
 is invertible,
so $\xi_n$ and $\zeta_n$ are well defined.

The definitions give
\[
 \begin{aligned}
 (C-\lambda_nI)\xi_n
 &=EWe_n+(\mu_n-\lambda_n)\xi_n,\\
 (C^*-\overline{\lambda_n}I)\zeta_n
 &=F^*(W^{-1})^*e_n
   +(\overline{\mu_n}-\overline{\lambda_n})\zeta_n.
 \end{aligned}
\]

Applying $(C-\lambda_nI)^{-1}$  to the first identity yields
\[
 \begin{aligned}
 \|\xi_n\|
 &\le\|(C-\lambda_nI)^{-1}\|
 \bigl(
 \|Ee_n\|+\|E\|\,\|We_n-e_n\|
       +|\mu_n-\lambda_n|\,\|\xi_n\|
 \bigr)\\
 &\le e^{cn^3}\min\{\eta,e^{-n^4}\}
       +\eta e^{-3n^3}
       +e^{-3n^3}\|\xi_n\|.
 \end{aligned}
\]
Moving the last term in each inequality to the left and using
$1-e^{-3n^3}>1/2$, we obtain
\[
 \|\xi_n\|
 \le 2e^{cn^3}\min\{\eta,e^{-n^4}\}
       +2\eta e^{-3n^3}.\]
The same argument applies to \(\zeta_n\).
These bounds are summable, since
$
 \sum_{n=1}^{\infty} e^{cn^3-n^4}<\infty.
$
It is routine to check that both matrices converge in trace norm.
\[
 \begin{aligned}
 \|\widetilde X\|_1
 &\le\sum_{i=1}^m\|\xi_i^h\|
       +\sum_{n=1}^{\infty}\|\xi_n\|,\\
 \|\widetilde Y\|_1
 &\le\sum_{i=1}^m\|\zeta_i^h\|
       +\sum_{n=1}^{\infty}\|\zeta_n\|.
 \end{aligned}
\]

  Choose $\eta=\eta(p,M,\beta)>0$ such that
\[
 4\eta\left(30+\sum_{n=1}^{\infty}e^{-3n^3}\right)
 +4\sum_{n=1}^{\infty}e^{cn^3}\min\{\eta,e^{-n^4}\}<\frac{\beta}{2}.
\]
 Then we have
 \[\|\widetilde X\|_1,
 \|\widetilde Y\|_1 <\frac{\beta}{4}.
\]

The defining equations give
\[
 (C-\mu_nI)\xi_n=EWe_n,
 \qquad
 \zeta_n^*(C-\mu_nI)=e_n^*W^{-1}F.
\]
Since $\widetilde Xe_n=\xi_n$ and
$e_n^*\widetilde Y=-\zeta_n^*$, we obtain
\[
 (C\widetilde X-\widetilde X\Lambda)e_n
 =(C-\mu_nI)\xi_n
 =EWe_n,
 \]\[
 e_n^*(\Lambda\widetilde Y-\widetilde YC)
 =-\mu_n\zeta_n^*+\zeta_n^*C
 =\zeta_n^*(C-\mu_nI)
 =e_n^*W^{-1}F.
\]
The same calculations hold for the first $m$ basis vectors.
Therefore,
\[
 C\widetilde X-\widetilde X\Lambda=EW,\qquad
 \Lambda\widetilde Y-\widetilde YC=W^{-1}F.
\]

Set
\[
 X=\widetilde XW^{-1},\qquad Y=W\widetilde Y.
\]
Since $A=W\Lambda W^{-1}$, it follows that
\[
 CX-XA=E,\qquad AY-YC=F.
\]
Moreover,
\[
 \|X\|_1\le2\|\widetilde X\|_1 <\frac{\beta}{2},\qquad
 \|Y\|_1\le2\|\widetilde Y\|_1<\frac{\beta}{2}.
\]
\end{proof}

We now pass to the limit in the finite-dimensional construction
and apply Lemma~\ref{lem:coupling} to obtain the required block
commutator representation.

Fix $p>2$ and $M\ge1$. Write
$\mathcal H_1=\mathbb C^m\oplus\ell^2(\mathbb N)$, where $1\le m\le3$.
Let $e$ be a unit vector in $\mathbb C^m$ and let $(e_n)_{n\ge1}$
be the standard basis of $\ell^2(\mathbb N)$. For $n\ge0$, let
$P_n$ be the orthogonal projection onto
\[
 \mathbb C^m\oplus\operatorname{span}\{e_1,\ldots,e_n\},
\]
where $P_0$ projects onto $\mathbb C^m$.
We consider $G\in \mathcal{S}_1(\mathcal H_1)$, bounded operators
$E:\mathcal H_1\to\mathcal H_0$ and $F:\mathcal H_0\to\mathcal H_1$,
and
\[
 C\in \mathcal{S}_p(\mathcal H_0),\qquad \|C\|\le1,\qquad \|C\|_p\le M.
\]

\begin{lemma}\label{lem:approximation}
Let $0<\beta\le1/8$, and take $\eta$ from Lemma~\ref{lem:coupling}
for $p,M,\beta$. For each nonzero trace-zero $D\in M_m(\mathbb C)$
with $16\sqrt3\,\|D\|_p<\beta$, there is a positive decreasing sequence
$(\varepsilon_n^*)_{n\ge0}$, depending only on $D,p,M,\beta$, with
the following property.
If $E,F$ satisfy \eqref{eq:EFhyp} and
\begin{equation*}\label{eq:Ghyp}
 \tr G=0,\qquad P_0GP_0-\tr(P_0GP_0)P_e=D,\qquad
 \|G-P_nGP_n\|_1\le\varepsilon_n^*\quad(n\ge0),
\end{equation*}
then there are $A_G,B_G\in \mathcal{S}_p(\mathcal H_1)$,
$X\in \mathcal{S}_1(\mathcal H_1,\mathcal H_0)$ and
$Y\in \mathcal{S}_1(\mathcal H_0,\mathcal H_1)$ such that
\[
 \left[\begin{pmatrix}C&0\\0&A_G\end{pmatrix},
 \begin{pmatrix}0&X\\Y&B_G\end{pmatrix}\right]
 =\begin{pmatrix}0&E\\F&G\end{pmatrix}
 \begin{matrix}\mathcal H_0\\\mathcal H_1\end{matrix},
\]
and
\begin{equation}\label{eq:complementarybounds} \tag{$**$}
 \begin{gathered}
 \left\|\begin{pmatrix}C&0\\0&A_G\end{pmatrix}\right\|
 \le\frac{5}{4}+\beta,\\
 \left\|\begin{pmatrix}C&0\\0&A_G\end{pmatrix}\right\|_p
 \le 3M+\beta,
 \qquad
 \left\|\begin{pmatrix}0&X\\Y&B_G\end{pmatrix}\right\|_p<3\beta.
 \end{gathered}
\end{equation}
Furthermore, if $G=0$, the same block
representation and estimates hold for every $1\le m\le3$.
\end{lemma}

\begin{proof}
Let $c=3(2M)^p$ as in Lemma~\ref{lem:circle}.
 Apply Lemma~\ref{lem:seed} with the interval $[9/8,5/4]$ to obtain a regular pair $(A_*,B_*)$ with $[A_*,B_*]=D$.  After a unitary change of coordinates on
$\mathbb C^m$, leaving $\ell^2(\mathbb N)$ fixed, we may assume
that $A_*$ is positive diagonal, with distinct entries
$\lambda_i^*\in[9/8,5/4]$ in an orthonormal basis $(e_i^h)_{i=1}^m$.
Moreover,
\[
 \|A_*\|_p\le\frac{5}{4}m^{1/p},\qquad
 \|B_*\|_p\le8(m-1)m^{1/2-1/p}\|D\|_p\leq 16\sqrt3\|D\|_p<\beta.
\]
Set
\[
 \begin{gathered}
 L_0=\|\mathcal R_{A_*,B_*}\|,\qquad
 K=3+\|A_*\|_p+\|B_*\|_p,\qquad
 g_0 =\operatorname{gap}(A_*),\\
 \rho_n=n^{-2/p},\qquad
 g_n=\min\left\{\frac{1}{16}, \frac{g_{n-1}}{2},
                         \frac{\rho_n-\rho_{n+1}}{2}\right\},\\
 \gamma_n=\frac{1}{2^{n+2}}\min\left\{
 \beta,g_n,
 e^{-(c+3)n^3}\right\}\quad(n\ge1).
 \end{gathered}
\]
The entries in the minimum are nonincreasing in $n$. Hence, for $n\ge1$,
\[
 \sum_{k=n}^{\infty}\gamma_k
 \le\frac{1}{4}\min\left\{
 \beta,\frac{g_0}{2},\frac{1}{16},\rho_n-\rho_{n+1},
 e^{-(c+3)n^3}\right\}.
\]

Starting with $L_0$, $\delta_0=1/(16L_0^2),$ choose $\delta_n$ recursively
by Lemma~\ref{lem:onestep}, using $\gamma_n$,
factor bound $K$,  right-inverse bound $1/(4\sqrt{\delta_{n-1}})$,
commutator lower bound $\|D\|/2$, and spectral-gap
lower bound $g_n$.
These choices depend only on $L_0,M,D,p,\beta$ and can
be made before $C$ and the points $\lambda_n$ are chosen.

Set
\[
 \varepsilon_n^*=
 \min\left\{\frac{1}{2^n},\frac{\|D\|}{4},
                  \frac{\delta_j}{2}:1\le j\le n+1\right\}
 \quad(n\ge0).
\]
This is a positive decreasing sequence tending to zero, chosen
before $C$ and its resolvent points.

Set
\[
 \widehat G_n=P_nGP_n-\tr(P_nGP_n)P_e.
\]
Then $\widehat G_0=D$, and Lemma~\ref{lem:trace-truncations} gives
\[
 \begin{gathered}
 \left\|\widehat G_n-
       \begin{pmatrix}\widehat G_{n-1}&0\\0&0\end{pmatrix}\right\|_p
 \le2\varepsilon_{n-1}^*\le\delta_n\quad(n\ge1),\\
 \|G-\widehat G_n\|_1\le2\varepsilon_n^*,\qquad
 \|\widehat G_n\|\ge\|D\|-2\varepsilon_0^*
                         \ge\frac{\|D\|}{2}.
 \end{gathered}
\]
The last estimate uses
$\|\widehat G_n-D\|\le\|\widehat G_n-D\|_1\le2\varepsilon_0^*$,
with $D$ extended by zero.

Now fix $C$. Lemma~\ref{lem:circle} supplies points $\lambda_n$ with
\[
 |\lambda_n|=\rho_n,\qquad
 \|(C-\lambda_nI)^{-1}\|\le e^{cn^3}.
\]
We construct regular pairs $[A_n,B_n]=\widehat G_n$ and invertible
$W_n$ such that $W_n^{-1}A_nW_n=\Lambda_n$ is diagonal. Start with
\[
 A_0=A_*,\qquad B_0=B_*,\qquad W_0=I_m,\qquad\Lambda_0=A_*.
\]
For $n=1$, the eigenvalues of $A_*$ are separated by $g_0$ and
have distance at least $1/8$ from $\lambda_1$. Thus
Lemma~\ref{lem:onestep} applies to $(A_*,B_*)$, $\lambda_1$ and
\[
 H=\widehat G_1-\begin{pmatrix}D&0\\0&0\end{pmatrix}.
\]
It gives $A_1,B_1,W_1,\Lambda_1$, with right-inverse norm at most
$1/(4\sqrt{\delta_1})$ and the four increments below bounded by $\gamma_1$.

For $n>1$, suppose the construction has been carried out for
$j<n$, with right-inverse norms at most \[
\|\mathcal R_{A_j,B_j}\|
\le\frac1{4\sqrt{\delta_j}}.
\] and
\[
 \begin{gathered}
 \left\|A_j-\begin{pmatrix}A_{j-1}&0\\0&\lambda_j\end{pmatrix}\right\|_p,
 \quad
 \left\|B_j-\begin{pmatrix}B_{j-1}&0\\0&0\end{pmatrix}\right\|_p
 \le\gamma_j,\\
 \left\|W_j-\begin{pmatrix}W_{j-1}&0\\0&1\end{pmatrix}\right\|,
 \quad
 \left\|\Lambda_j-
       \begin{pmatrix}\Lambda_{j-1}&0\\0&\lambda_j\end{pmatrix}\right\|
 \le\gamma_j.
 \end{gathered}
\]
These matrices use the decomposition
$P_j\mathcal H_1=P_{j-1}\mathcal H_1\oplus\mathbb Ce_j$.

We now proceed to justify the inductive step. It remains to show that
\begin{equation*}
\operatorname{gap}
\begin{pmatrix}\Lambda_{n-1}&0\\0&\lambda_n\end{pmatrix}
\ge g_n,\qquad n\geq 1.
\end{equation*}
The above inequality holds for $n=1$ by the preceding argument.
For $n\geq 2$, suppose that the inequality holds for $1\le k\le n-1$.
The obtained bound establishes
\[
\left\|\Lambda_{n-1}-
\begin{pmatrix}\Lambda_{n-2}&0\\0&\lambda_{n-1}\end{pmatrix}\right\|
\le \gamma_{n-1}
\le \frac{1}{4}g_{n-1}.
\]
Thus Proposition~\ref{prop:spectral-gap} applies and yields
\[
\operatorname{gap}(\Lambda_{n-1})
\ge
g_{n-1}
-2\gamma_{n-1}
\geq g_n.
\]

We next estimate the distance from $\lambda_n$ to
$\sigma(\Lambda_{n-1})$. Write
\[
\Lambda_{n-1}
=\operatorname{diag}(H^{(n-1)},\mu_1^{(n-1)},\ldots,\mu_{n-1}^{(n-1)}),
\qquad H^{(n-1)}\in M_m(\mathbb C).
\]
Summing the diagonal increments directly yields
\[
\|H^{(n-1)}-\Lambda_0\|
\le\sum_{k=1}^{n-1}\gamma_k
\le\frac{g_1}{4},
\qquad
|\mu_i^{(n-1)}-\lambda_i|
\le\sum_{k=i}^{n-1}\gamma_k
\le\frac{g_i}{4}.
\]
Note that $\sigma(\Lambda_0)\subset [9/8,5/4]$, while
$|\lambda_n|\le1$. Moreover, for $1\le i<n$,
\[
|\lambda_i-\lambda_n|
\ge\rho_i-\rho_n
\ge\rho_i-\rho_{i+1}
\ge2g_i.
\]
Consequently,
\[
\begin{aligned}
\operatorname{dist}\bigl(\lambda_n,\sigma(\Lambda_{n-1})\bigr)
&=
\min\left\{
\operatorname{dist}\bigl(\lambda_n,\sigma(H^{(n-1)})\bigr),
\min_{1\le i<n}|\mu_i^{(n-1)}-\lambda_n|
\right\}\\
&\ge
\min\left\{
\frac18-\frac{g_1}{4},
\min_{1\le i<n}
\left|2g_i-\frac{g_i}{4}\right|
\right\}\\
&\geq g_n.
\end{aligned}
\]
Here we used $g_1\le1/16$ and the monotonicity of $(g_j)$.

Now we obtain
\[
\operatorname{gap}
\begin{pmatrix}\Lambda_{n-1}&0\\0&\lambda_n\end{pmatrix}
=
\min\left\{
\operatorname{gap}(\Lambda_{n-1}),
\operatorname{dist}\bigl(\lambda_n,\sigma(\Lambda_{n-1})\bigr)
\right\}
\ge g_n.
\]

Also, telescoping the increments yields
\[
 \begin{aligned}
 \|A_{n-1}\|_p
 &\le\left(\|A_*\|_p^p+\sum_{j=1}^{n-1}|\lambda_j|^p\right)^{1/p}
       +\sum_{j=1}^{n-1}\gamma_j<K,\\
 \|B_{n-1}\|_p&\le\|B_*\|_p+\sum_{j=1}^{n-1}\gamma_j<K,\\
 \|W_{n-1}-I\|&\le\sum_{j=1}^{n-1}\gamma_j<\frac{1}{4}.
 \end{aligned}
\]
Here $\sum_j|\lambda_j|^p=\zeta(2)<2$ and
$\sum_j\gamma_j\le\beta/4$. The commutator norm is at least
$\|D\|/2$, and the right-inverse norm is at most $1/(4\sqrt{\delta_{n-1}})$.
We may therefore apply Lemma~\ref{lem:onestep} with
\[
 H=\widehat G_n-
       \begin{pmatrix}\widehat G_{n-1}&0\\0&0\end{pmatrix},
 \qquad\tr H=0,\qquad\|H\|_p\le\delta_n.
\]
This gives the next pair and diagonalization, with right-inverse norm
at most $1/(4\sqrt{\delta_{n}})$ and all four increments at most $\gamma_n$.

Extend $A_n,B_n,\Lambda_n$ by zero and $W_n$ by the identity on
$(P_n\mathcal H_1)^\perp$. For $N>n$, the same estimates give
\[
 \begin{aligned}
 \left\|\begin{pmatrix}A_N&0\\0&0\end{pmatrix}
           -\begin{pmatrix}A_n&0\\0&0\end{pmatrix}\right\|_p
 &\le\big(\sum_{j=n+1}^N|\lambda_j|^p\big)^{1/p}
       +\sum_{j=n+1}^N\gamma_j,\\
 \left\|\begin{pmatrix}B_N&0\\0&0\end{pmatrix}
           -\begin{pmatrix}B_n&0\\0&0\end{pmatrix}\right\|_p
 &\le\sum_{j=n+1}^N\gamma_j,\\
 \left\|\begin{pmatrix}W_N&0\\0&I\end{pmatrix}
           -\begin{pmatrix}W_n&0\\0&I\end{pmatrix}\right\|
 &\le\sum_{j=n+1}^N\gamma_j.
 \end{aligned}
\]
Each block matrix with subscript $j$ acts on
$P_j\mathcal H_1\oplus(P_j\mathcal H_1)^\perp$.
The right sides tend to zero uniformly as $n\to\infty$.

Define
$$
A_G=\lim_{n\to \infty}\begin{pmatrix}A_{n}&\\&0\end{pmatrix},\quad
B_G=\lim_{n\to \infty}\begin{pmatrix}B_{n}&\\&0\end{pmatrix},
\quad W=\lim_{n\to \infty}\begin{pmatrix}W_{n}&\\&I\end{pmatrix}.
$$
Then $A_G,B_G \in \mathcal{S}_p(\mathcal{H}_1)$, and $\|W-I\|\leq1/4$, so
$W$ is invertible and $\|W\|,\|W^{-1}\|<2$.
In comparisons on $\mathcal H_1$, retain the same notation for the
extended finite matrices. H\"older's inequality gives
\[
 \|[A_n,B_n]-[A_G,B_G]\|_{p/2}
 \le2\|A_n-A_G\|_p\|B_n\|_p
     +2\|A_G\|_p\|B_n-B_G\|_p\longrightarrow0.
\]
Since $\widehat G_n\to G$ in trace norm, it also converges in
$\mathcal{S}_{p/2}$, and therefore
$$[A_G,B_G]=G.$$

It remains to verify the assumptions on $A_G$ and $W$ in
Lemma~\ref{lem:coupling}.
Since $W$ is invertible, the inverse map is also norm continuous. Hence
\[
 \begin{pmatrix}\Lambda_n&0\\0&0\end{pmatrix}
 =\begin{pmatrix}W_n^{-1}&0\\0&I\end{pmatrix}
 \begin{pmatrix}A_n&0\\0&0\end{pmatrix}
 \begin{pmatrix}W_n&0\\0&I\end{pmatrix}
 \longrightarrow W^{-1}A_GW=:\Lambda.
\]
Each $\Lambda_n$ is diagonal in the fixed basis
$e_1^h,\ldots,e_m^h,e_1,e_2,\ldots$, so $\Lambda$ is diagonal
in this basis. Denote its diagonal entries by
$\mu_1^h,\ldots,\mu_m^h,\mu_1,\mu_2,\cdots$. Then we have
\[
\left\|
\Lambda-
\begin{pmatrix}
\Lambda_{n-1}&0&0&\cdots\\
0&\lambda_n&0&\cdots\\
0&0&\lambda_{n+1}&\ddots\\
\vdots&\vdots&\ddots&\ddots
\end{pmatrix}
\right\|
\le\sum_{k=n}^{\infty}\gamma_k.
\]
Since $\sigma(\Lambda_0)\subset [9/8,5/4]$, for $1\le i\le m$ we have
\[
 |\mu_i^h|\ge\frac{9}{8}-\sum_{k=1}^{\infty}\gamma_k
 \ge\frac{9}{8}-\frac{1}{64}>\frac{11}{10}.
\]
For $n\ge1$,
\[
|\mu_n-\lambda_n|
\le\sum_{k=n}^{\infty}\gamma_k\leq\frac{1}{4}e^{-(c+3)n^3}.
\]

 Using the fact that
 \[
 W^{-1}-\begin{pmatrix}W_{n-1}&0\\0&I\end{pmatrix}^{-1}
 =W^{-1}\left(\begin{pmatrix}W_{n-1}&0\\0&I\end{pmatrix}-W\right)
 \begin{pmatrix}W_{n-1}&0\\0&I\end{pmatrix}^{-1},
\]
we obtain
\[
 \begin{aligned}
 \|We_n-e_n\|&\le\sum_{k=n}^{\infty}\gamma_k
                       \le e^{-(c+3)n^3},\\
 \|(W^{-1})^*e_n-e_n\|&\le4\sum_{k=n}^{\infty}\gamma_k
                       \le e^{-(c+3)n^3}.
 \end{aligned}
\]
This verifies all the required assumptions on $A_G$ and $W$.

If $G=0$, choose the same resolvent points $\lambda_n$
and set
\[
 A_G=\frac98 I_{\mathbb C^m}\oplus
      \operatorname{diag}(\lambda_1,\lambda_2,\ldots),
 \qquad B_G=0,\qquad W=I.
\]
Then $[A_G,B_G]=G$, and the factor bounds and the assumptions of
Lemma~\ref{lem:coupling} hold directly.

Apply Lemma~\ref{lem:coupling} to $A_G$ and $E,F$.
It gives $X,Y$ with $\|X\|_1,\|Y\|_1<\beta/2$
such that
$$
CX-XA_G=E,\qquad  A_GY-YC=F.$$ Thus,
\[
 \left[\begin{pmatrix}C&0\\0&A_G\end{pmatrix},
 \begin{pmatrix}0&X\\Y&B_G\end{pmatrix}\right]
 =\begin{pmatrix}0&E\\F&G\end{pmatrix}
 \begin{matrix}\mathcal H_0\\\mathcal H_1\end{matrix}.
\]
In the nonzero case,
$\|B_G\|_p\le\|B_*\|_p+\sum_{n\ge1}\gamma_n<2\beta$;
in the zero case $B_G=0$. Thus, in both cases,
\begin{align*}
 \left\|\begin{pmatrix}0&X\\Y&B_G\end{pmatrix}\right\|_p
 &\le\|B_G\|_p+\|X\|_1+\|Y\|_1<3\beta.
\end{align*}

Since $M\geq 1$, $\zeta(2)<2$, we have
\[
 \left\|\begin{pmatrix}C&0\\0&A_G\end{pmatrix}\right\|
 \le \max\{1,\frac{5}{4}+\sum_{n=1}^{\infty}\gamma_n\}<\frac{5}{4}+\beta,
 \]
 \begin{align*}
     \left\|\begin{pmatrix}C&0\\0&A_G\end{pmatrix}\right\|_p &
     \leq \left(M^p+m\left(\frac{5}{4}\right)^p+\zeta(2)\right)^{1/p}+\beta \\
      & \leq M(3+m(\frac{5}{4})^p)^{1/p}+\beta < 3M+\beta.
   \end{align*}
Indeed, $m\le3$ and monotonicity of finite-dimensional $\ell_p$-norms for $p>2$ give
\[
 \bigl(3+m(5/4)^p\bigr)^{1/p}
 \le\bigl(3+3(5/4)^p\bigr)^{1/p}
 \le\left(3+3\left(\frac54\right)^2\right)^{1/2}
 =\frac{\sqrt{123}}4<3.
\]

\end{proof}

\section{Main results}

We combine the decomposition of Section~\ref{sec:geometry} with
Lemma~\ref{lem:approximation} to obtain the preparation lemma below.
The remaining compression has an infinite-dimensional reducing null
subspace, so Theorem~\ref{lem:stable} supplies its commutator representation.
Combining the two representations proves Theorem~\ref{thm:main}
with the stated norm estimates.
\begin{lemma}\label{lem:preparation}
Let $T$ be a compact operator of infinite rank with $\|T\|=1$.
Fix $p>2$, $M\ge1$, $\gamma>0$ and $0<\beta\le 1/8$.
There are an orthogonal decomposition
$\mathcal H=\mathcal H_0\oplus\mathcal H_1$ and an invertible
operator $S$ with $\|S\|\|S^{-1}\|<1+\gamma$ such that
\[
 S^{-1}TS=
 \begin{pmatrix}
 \widehat D&E\\
 F&G
 \end{pmatrix}
 \begin{matrix}
 \mathcal H_0\\
 \mathcal H_1
 \end{matrix},
\]
where $\widehat D=P_{\mathcal H_0}S^{-1}TS|_{\mathcal H_0}$ has an infinite-dimensional reducing null subspace.

Furthermore, for every $C\in \mathcal{S}_p(\mathcal H_0)$ satisfying
$\|C\|\le1$ and $\|C\|_p\le M$, there are
$A_G,B_G\in \mathcal{S}_p(\mathcal H_1)$,
$X\in \mathcal{S}_1(\mathcal H_1,\mathcal H_0)$ and
$Y\in \mathcal{S}_1(\mathcal H_0,\mathcal H_1)$ satisfying
\eqref{eq:complementarybounds} and
\begin{equation*}\label{eq:preparedidentity}
 \left[
 \begin{pmatrix}
 C&0\\
 0&A_G
 \end{pmatrix},
 \begin{pmatrix}
 0&X\\
 Y&B_G
 \end{pmatrix}
 \right]
 =
 \begin{pmatrix}
 0&E\\
 F&G
 \end{pmatrix}
 \begin{matrix}
 \mathcal H_0\\
 \mathcal H_1
 \end{matrix}.
\end{equation*}
\end{lemma}
\begin{proof}
Take $\eta=\eta(p,M,\beta)$ from Lemma~\ref{lem:coupling}, as used in Lemma~\ref{lem:approximation}.
Choose a positive decreasing sequence $(\varepsilon_n)$ such that
\[
 \sum_{n=1}^{\infty}\varepsilon_n<\frac{\eta}{4},\qquad
 2\varepsilon_n\le e^{-(2n)^4},\qquad
 8\varepsilon_1<\frac{\beta}{16\sqrt3}.
\]
Apply Lemma~\ref{lem:shear} for $\min\{\gamma,\eta/2\}$ (in place of $\gamma$) and  $(\varepsilon_n)$. It gives
\[
 S^{-1}TS=
  \begin{pmatrix}
 \widehat D&E_v&E_1&E_2&\cdots\\
 F_v&\lambda&0&0&\cdots\\
 F_1&0&G_1&0&\cdots\\
 F_2&0&0&G_2&\cdots\\
 \vdots&\vdots&\vdots&\vdots&\ddots
 \end{pmatrix}
 \begin{matrix}
 \mathcal H_0\\\mathbb Cv\\\mathcal H_{1,1}\\\mathcal H_{1,2}\\\vdots
 \end{matrix}
 \,\triangleq\,\begin{pmatrix}\widehat D&E\\F&G\end{pmatrix}
 \begin{matrix}\mathcal H_0\\\mathcal H_1\end{matrix},
\]
where $\lambda=-\sum_n\tr G_n$,
$1\le\dim\mathcal H_{1,n}\le2$, and
$\widehat D$ has an infinite-dimensional reducing null subspace.

Moreover,
\[
 \|G_n\|_1\le4\varepsilon_n,\qquad
 \|E_n\|\le2\varepsilon_n,\qquad \|F_n\|\le2\varepsilon_n,
\]
\[
\|S\|\|S^{-1}\|<1+\min\{\gamma,\eta/2\},
 \qquad \|E_v\|<\min\{\gamma,\eta/2\},
 \qquad \|F_v\|<\min\{\gamma,\eta/2\}.
\]
Then
 \begin{align*}
 \|E\|&\le\|E_v\|+\sum_{n=1}^{\infty}\|E_n\|
       <\frac{\eta}{2}+2\sum_{n=1}^{\infty}\varepsilon_n<\eta,\\
 \|F\|&\le\|F_v\|+\sum_{n=1}^{\infty}\|F_n\|
       <\frac{\eta}{2}+2\sum_{n=1}^{\infty}\varepsilon_n<\eta.
 \end{align*}

If every $G_n$ is zero, then $G=0$. Take $\mathbb Cv$ as the
one-dimensional first summand of $\mathcal H_1$, and enumerate bases
of the spaces $\mathcal H_{1,j}$ successively as $(e_n)$.
For $e_n\in\mathcal H_{1,j}$, we have $n\le2j$, and hence
\[
 \|Ee_n\|,\ \|F^*e_n\|
 \le \min\{\eta,2\varepsilon_j\}\le \min\{\eta,e^{-(2j)^4}\}\le \min\{\eta,e^{-n^4}\}.
\]
Thus the $G=0$ case of Lemma~\ref{lem:approximation} applies to this
decomposition for every admissible $C$.

Suppose that $G\neq 0$, and let
\[
 { k}=\min\{n\ge1:G_n\ne0\}.
\]
Set
\[
 \mathcal H_{\mathrm h}=\mathbb Cv\oplus\mathcal H_{1,{k}},
 \qquad m=\dim\mathcal H_{\mathrm h},
 \qquad
 D=
 \begin{pmatrix}
 -\tr G_{ k}&0\\
 0&G_{ k}
 \end{pmatrix}
 \begin{matrix}
 \mathbb Cv\\
 \mathcal H_{1,{ k}}
 \end{matrix}.
\]
Then $2\le m\le3$, $D\ne0$, $\tr D=0$, and
\[
 \begin{aligned}
 \|D\|_p
 &\le\|D\|_1
 =|\tr G_{ k}|+\|G_{ k}\|_1\\
 &\le2\|G_{ k}\|_1
 \le8\varepsilon_{ k}
 \le8\varepsilon_1
 <\frac{\beta}{16\sqrt3}.
 \end{aligned}
\]
Now Lemma~\ref{lem:approximation} supplies a positive decreasing sequence $(\varepsilon_n^*)$ for this fixed $D$.

Define a smaller decreasing sequence $(\varepsilon_n')$ by
\[
 \varepsilon_n'=
 \begin{cases}
    \varepsilon_n, & \mbox{if } n\le k\\
   \min\{
 \varepsilon_{n},\,2^{k-n-4}\varepsilon_{2(n-k)}^*\}, & \mbox{if } n>k.
 \end{cases}
 \]
According to Remark~\ref{rem:sheartail}, we can apply Lemma~\ref{lem:shear} for $\min\{\gamma,\eta/2\}$ (in place of $\gamma$) and $(\varepsilon_n')$. Denote the resulting similarity, vectors, spaces, and blocks by \(S',f_n',\mathcal H_{1,n}',G_n'\), with
\[
 \|S'\|\|S'^{-1}\|<1+\gamma,
\qquad
(f_i',\mathcal H_{1,i}',G_i')=(f_i,\mathcal H_{1,i},G_i)
\quad(i\le k).
\]
The fixed matrix \(D\) remains unchanged.

Let
\[
 \mathcal H_{\mathrm t}'
 =\bigoplus_{j\ge1}\mathcal H_{1,k+j}',
 \qquad
 \mathcal H_1'=\mathcal H_{\mathrm h}\oplus\mathcal H_{\mathrm t}',
 \qquad
 \mathcal H_0'=(\mathcal H_1')^\perp.
\]
Define the new compressions by
\[
 S'^{-1}TS'
 =
 \begin{pmatrix}
 \widehat D'&E'\\
 F'&G'
 \end{pmatrix}
 \begin{matrix}
 \mathcal H_0'\\
 \mathcal H_1'
 \end{matrix}.
\]
The preceding block calculation gives
\[
 G'=
 \begin{pmatrix}
 D-(\tr H')P_v&0\\
 0&H'
 \end{pmatrix}
 \begin{matrix}
 \mathcal H_{\mathrm h}\\
 \mathcal H_{\mathrm t}'
 \end{matrix},
 \qquad
 H'=
 \begin{pmatrix}
 G_{k+1}'&0&\cdots\\
 0&G_{k+2}'&\ddots\\
 \vdots&\ddots&\ddots
 \end{pmatrix}
 \begin{matrix}
 \mathcal H_{1,k+1}'\\
 \mathcal H_{1,k+2}'\\
 \vdots
 \end{matrix}.
\]
In particular,
\[
\|G_{k+j}'\|_1\leq  4\varepsilon_{k+j}'\leq  2^{-j-2}\varepsilon_{2j}^*,
\]\[
 \|H'\|_1
 =\sum_{j=1}^{\infty}\|G_{k+j}'\|_1
 \le\sum_{j=1}^{\infty}2^{-j-2}\varepsilon_{2j}^*<\infty,
\]

For $n>k$, Lemma~\ref{lem:shear} gives orthonormal vectors
$f_n'\in\mathcal H_0'$ such that
\[
 S'^{-1}TS'f_n',\quad (S'^{-1}TS')^*f_n'
 \in\mathcal H_{1,n}'\subseteq\mathcal H_1'.
\]
Consequently, $\widehat D'$ has an infinite-dimensional reducing null subspace.

We verify the hypotheses of Lemma~\ref{lem:approximation}.
For each $j\ge1$, choose an orthonormal basis of
$\mathcal H_{1,k+j}'$. Enumerate these bases successively
as $(e_n)_{n\ge1}$.
For $n\ge0$, let $P_n'$ project onto
\[
 \mathcal H_{\mathrm h}
 \oplus\operatorname{span}\{e_1,\ldots,e_n\}.
\]
Thus $P_0'$ projects onto $\mathcal H_{\mathrm h}$ and
$\operatorname{rank}P_n'=m+n$.
Regard $D,P_v$ and the blocks $G_{k+j}'$ as operators on
$\mathcal H_1'$ by zero extension. The matrix of $G'$ gives
\[
 \begin{aligned}
 \tr G'
 &=\tr D-\tr H'+\tr H'=0,\\
 P_0'G'P_0'-\tr(P_0'G'P_0')P_v
 &=D-(\tr H')P_v+(\tr H')P_v=D.
 \end{aligned}
\]

Since each tail block has dimension at most two, we have
\[
 P_n'G_{k+j}'P_n'=G_{k+j}'
 \qquad(2j\le n).
\]
Consequently, for every $n\ge0$,
\[
 \begin{aligned}
 \|G'-P_n'G'P_n'\|_1
 &\le
 \sum_{2j>n}
 \|G_{k+j}'-P_n'G_{k+j}'P_n'\|_1\\
 &\le
 2\sum_{2j>n}\|G_{k+j}'\|_1\\
 &\le
 8\sum_{2j>n}
       2^{-j-4}\varepsilon_{2j}^*\\
 &\le
 8\varepsilon_n^*\sum_{j\ge1}2^{-j-4}
 <\varepsilon_n^*.
 \end{aligned}
\]

The spaces moved into $\mathcal H_0'$ have zero blocks to
$\mathcal H_1'$. Thus the coupling estimates from
Lemma~\ref{lem:shear} give
\[
 \|E'\|,\ \|F'\|
 \le\frac{\eta}{2}+2\sum_{i\ge k}\varepsilon_i'
 <\eta.
\]
If $e_n\in\mathcal H_{1,k+j}'$, then $n\le2j$, and
\[
 \|E'e_n\|,\ \|(F')^*e_n\|\le2\varepsilon_{k+j}'
\le2\varepsilon_{k+j}
 \le e^{-(2(k+j))^4}
 \le e^{-n^4}.
\]
This verifies the required conditions on $G',E',F'$.

The spaces and $S'$ have been fixed independently of $C$.
Now let $C\in \mathcal{S}_p(\mathcal H_0')$ satisfy
$\|C\|\le1$ and $\|C\|_p\le M$.
Apply Lemma~\ref{lem:approximation} to
$G',E',F'$ with first summand
$\mathcal H_{\mathrm h}\cong\mathbb C^m$ and $e=v$.
There are operators
\[A_{G'},B_{G'}\in \mathcal{S}_p(\mathcal H_1'),\qquad
 X\in \mathcal{S}_1(\mathcal H_1',\mathcal H_0'),
 \qquad
 Y\in \mathcal{S}_1(\mathcal H_0',\mathcal H_1'),
\]
such that
\[
 \left[
 \begin{pmatrix}
 C&0\\
 0&A_{G'}
 \end{pmatrix},
 \begin{pmatrix}
 0&X\\
 Y&B_{G'}
 \end{pmatrix}
 \right]
 =
 \begin{pmatrix}
 0&E'\\
 F'&G'
 \end{pmatrix}
 \begin{matrix}
 \mathcal H_0'\\
 \mathcal H_1'
 \end{matrix}
 =
 S'^{-1}TS'-
 \begin{pmatrix}
 \widehat D'&0\\
 0&0
 \end{pmatrix}.
\]
Rewriting in the original notation yields the desired result.
\end{proof}

\begin{proof}[{\bf Proof of Theorem~\ref{thm:main}}]
Let $r>2$ and let $p,q\in(0,\infty)$ satisfy
$r^{-1}=p^{-1}+q^{-1}.$
Without loss of generality, we assume $p\le q$.

If $T$ has finite rank, then $(\operatorname{ran}T+\operatorname{ran}T^*)^\perp$ is an
infinite-dimensional reducing null subspace. Theorem~\ref{lem:stable} gives
\[
 T=[A,B],\qquad A\in \mathcal{S}_p(\mathcal H),\quad B\in \mathcal{S}_q(\mathcal H),
\]
with
\[
 \|A\|\|B\|\le4\|T\|,
 \qquad
 \|A\|_p\|B\|_q
 \le4\zeta(r-1)^{1/r}\|T\|_r.
\]

Now suppose that $T$ has infinite rank and $\|T\|=1$. Set
\[
 M_p=\zeta(r-1)^{1/p}\|T\|_r^{r/p},
 \qquad
 M_q=\zeta(r-1)^{1/q}\|T\|_r^{r/q}.
\]
Then $M_p,M_q\ge1$ and
$
 M_p M_q=\zeta(r-1)^{1/r}\|T\|_r.
$
Choose $\gamma>0$ and $0<\beta\le1/8$ sufficiently small that
\[
 (1+\gamma)^2
 \left(\frac54+\beta\right)
 \bigl(4(1+\gamma)+3\beta\bigr)<6
\]
and
\[
 (1+\gamma)^2
 (3M_p+\beta)
 \bigl(4(1+\gamma)M_q+3\beta\bigr)<16M_p M_q.
\]
Such a choice is possible because the left-hand sides tend to
$5$ and $12M_p M_q$, respectively, as $\gamma,\beta\to 0$.

Apply Lemma~\ref{lem:preparation} with $p,M_p,\gamma,\beta$.
It gives an orthogonal decomposition
$\mathcal H=\mathcal H_0\oplus\mathcal H_1$ and an invertible
operator $S$ such that
\[
 \|S\|\|S^{-1}\|<1+\gamma,
 \qquad
 S^{-1}TS=
 \begin{pmatrix}\widehat D&E\\F&G\end{pmatrix},
\]
where $\widehat D$ has an infinite-dimensional reducing null
subspace. Moreover,
\[
 \|\widehat D\|\le1+\gamma,
 \qquad
 \|\widehat D\|_r\le(1+\gamma)\|T\|_r.
\]
By Theorem~\ref{lem:stable},  there exist $C_0\in \mathcal{S}_p(\mathcal H_0)$ and $Z_0\in \mathcal{S}_q(\mathcal H_0)$ such that
$\widehat D=[C_0,Z_0]$ and
\[
 \begin{aligned}
 \|C_0\|&\le2(1+\gamma)^{r/p},
 &\qquad \|C_0\|_p&\le2(1+\gamma)^{r/p}M_p,\\
 \|Z_0\|&\le2(1+\gamma)^{r/q},
 &\qquad \|Z_0\|_q&\le2(1+\gamma)^{r/q}M_q.
 \end{aligned}
\]
Set
\[
 C=\frac{C_0}{2(1+\gamma)^{r/p}},
 \qquad
 Z=2(1+\gamma)^{r/p}Z_0.
\]
Since $r/p+r/q=1$, these operators satisfy
\[
 [C,Z]=\widehat D,\qquad
 \|C\|\le1,\qquad \|C\|_p\le M_p,
\]
and
\[
 \|Z\|\le4(1+\gamma),
 \qquad
 \|Z\|_q\le4(1+\gamma)M_q.
\]

Use the conclusion of Lemma~\ref{lem:preparation} for this $C$
to obtain $A_G,B_G\in \mathcal{S}_p(\mathcal H_1)$ and
\[
 X\in \mathcal{S}_1(\mathcal H_1,\mathcal H_0),
 \qquad
 Y\in \mathcal{S}_1(\mathcal H_0,\mathcal H_1).
\]
Let
\[
 \widetilde A=
 \begin{pmatrix}C&0\\0&A_G\end{pmatrix},
 \qquad
 \widetilde B=
 \begin{pmatrix}Z&X\\Y&B_G\end{pmatrix}.
\]
Then
\[
 \begin{aligned}
 [\widetilde A,\widetilde B]
 &=
 \begin{pmatrix}
 [C,Z]&CX-XA_G\\
 A_GY-YC&[A_G,B_G]
 \end{pmatrix}\\
 &=
 \begin{pmatrix}\widehat D&E\\F&G\end{pmatrix}
 =S^{-1}TS.
 \end{aligned}
\]

The estimates in \eqref{eq:complementarybounds} give
\[
 \|\widetilde A\|\le\frac54+\beta,
 \qquad
 \|\widetilde A\|_p\le3M_p+\beta,
 \qquad
 \left\|\begin{pmatrix}0&X\\Y&B_G\end{pmatrix}\right\|_p
 <3\beta.
\]
Since $p\le q$, the last block operator also has
$\mathcal{S}_q$-norm less than $3\beta$. Therefore
$\widetilde A\in \mathcal{S}_p(\mathcal H)$,
$\widetilde B\in \mathcal{S}_q(\mathcal H)$, and
\[
 \|\widetilde B\|<4(1+\gamma)+3\beta,
 \qquad
 \|\widetilde B\|_q<4(1+\gamma)M_q+3\beta.
\]
Set
\[
 A=S\widetilde AS^{-1},
 \qquad B=S\widetilde BS^{-1}.
\]
Then $A\in \mathcal{S}_p(\mathcal H)$, $B\in \mathcal{S}_q(\mathcal H)$ and
$[A,B]=T$. The choice of $\gamma,\beta$
yields
\[
 \begin{aligned}
 \|A\|\|B\|
 &<(1+\gamma)^2
   \left(\frac54+\beta\right)
   \bigl(4(1+\gamma)+3\beta\bigr)<6,\\
 \|A\|_p\|B\|_q
 &<(1+\gamma)^2
   (3M_p+\beta)
   \bigl(4(1+\gamma)M_q+3\beta\bigr)\\
 &<16M_p M_q
 =16\zeta(r-1)^{1/r}\|T\|_r.
 \end{aligned}
\]

Finally, for an arbitrary nonzero $T$ of infinite rank, apply
the normalized construction to $T/\|T\|$, obtaining factors
$A_0,B_0$, and set
\[
 A=\|T\|^{r/p}A_0,
 \qquad B=\|T\|^{r/q}B_0.
\]
Again using $r/p+r/q=1$, we obtain $[A,B]=T$ and
\[
 \|A\|\|B\|<6\|T\|,
 \qquad
 \|A\|_p\|B\|_q
 <16\zeta(r-1)^{1/r}\|T\|_r.
\]
Together with the finite-rank case and the symmetry in $p,q$,
this proves the theorem.
\end{proof}

Brown's result \cite{Brown} already establishes the necessity of the condition $r>2$. His idea also suggests  a quantitative lower
bound for $\|A\|_p\|B\|_q$ when $[A,B]$ is a rank-one orthogonal projection
; see also
\cite[proof of Theorem~7.3]{DFWW}.

\begin{proposition}\label{prop:lower}
Let $1\leq r,p,q<\infty$ satisfy
$r^{-1}=p^{-1}+q^{-1}$.  Let $P$ be a rank-one orthogonal projection.
If $P=[A,B]$ with $A\in \mathcal{S}_p(\mathcal{H})$ and
$B\in \mathcal{S}_q(\mathcal{H})$, then $r>2$ and
\begin{equation*}
\|A\|_p\|B\|_q
\geq\frac12\zeta(r-1)^{1/r}.
\end{equation*}
\end{proposition}
\begin{proof}
The conclusion $r>2$ follows directly from Brown's theorem \cite{Brown}. Applying the trace computation in his proof, with the initial subspace $\operatorname{ran} P$, yields a sequence of mutually orthogonal finite-rank projections $\{P_n\}_{n \ge 0}$ such that
\begin{equation*}
P_0=P,
\qquad
\dim\operatorname{ran}P_n\leq n+1,
\end{equation*}
and
\begin{equation*}\label{eq:rankone-boundary-trace}
1=\operatorname{Tr}(P_nAP_{n+1}BP_n)
-\operatorname{Tr}(P_nBP_{n+1}AP_n)
\qquad(n\geq0).
\end{equation*}
The block estimates are 
\begin{equation*}
    \|(\|P_nAP_{n+1}\|_p)_{n=0}^\infty\|_{\ell^p}\leq \|A\|_p,\quad  \|(\|P_nBP_{n+1}\|_q)_{n=0}^\infty\|_{\ell^q}\leq \|B\|_q.
\end{equation*}
Similarly, applying these estimates to $A^*$ and $B^*$ gives
\begin{equation*}
    \|(\|P_{n+1}AP_{n}\|_p)_{n=0}^\infty\|_{\ell^p}\leq \|A\|_p,\quad  \|(\|P_{n+1}BP_{n}\|_q)_{n=0}^\infty\|_{\ell^q}\leq \|B\|_q.
\end{equation*}

Set $s=r/(r-1)$, so $r^{-1}+s^{-1}=1$.
Note that for every $T=P_nTP_n$, Proposition~\ref{prop:norms} gives
\begin{equation*}
|\operatorname{Tr}T|
=|\operatorname{Tr}(P_nT)|
\leq \|P_nT\|_1
\leq\|P_n\|_s\|T\|_r
\leq(n+1)^{1-1/r}\|T\|_r.
\end{equation*}
Applying this estimate to $T=P_nAP_{n+1}BP_n-P_nBP_{n+1}AP_n$ gives
\begin{equation*}
\begin{aligned}
(n+1)^{1/r-1}
&\leq\|P_nAP_{n+1}BP_n\|_r
+\|P_nBP_{n+1}AP_n\|_r\\
&\leq\|P_nAP_{n+1}\|_p\|P_{n+1}BP_n\|_q+\|P_nBP_{n+1}\|_q\|P_{n+1}AP_n\|_p.
\end{aligned}
\end{equation*}

By Minkowski's inequality and H\"older's
inequality,
\begin{equation*}
\begin{aligned}
\left(\sum_{n=0}^{N}(n+1)^{1-r}\right)^{1/r}
&\leq \left(\sum_{n=0}^{N} (\|P_nAP_{n+1}\|_p\|P_{n+1}BP_n\|_q+\|P_nBP_{n+1}\|_q\|P_{n+1}AP_n\|_p)^r\right)^{1/r}\\
&\leq \left(\sum_{n=0}^{N} (\|P_nAP_{n+1}\|_p\|P_{n+1}BP_n\|_q)^r  \right)^{1/r}\\
&\quad+\left(\sum_{n=0}^{N} (\|P_nBP_{n+1}\|_q\|P_{n+1}AP_n\|_p)^r  \right)^{1/r}\\
&\leq
\left(
\sum_{n=0}^{N}\|P_nAP_{n+1}\|_p^p
\right)^{1/p}
\left(
\sum_{n=0}^{N}\|P_{n+1}BP_n\|_q^q
\right)^{1/q}\\
&\quad+
\left(
\sum_{n=0}^{N}\|P_nBP_{n+1}\|_q^q
\right)^{1/q}
\left(
\sum_{n=0}^{N}\|P_{n+1}AP_n\|_p^p
\right)^{1/p}\\
&\leq2\|A\|_p\|B\|_q,
\end{aligned}
\end{equation*}
Then we obtain
\begin{equation*}
\zeta(r-1)^{1/r}
=\left(\sum_{n=0}^{\infty}(n+1)^{1-r}\right)^{1/r}
\leq2\|A\|_p\|B\|_q.
\end{equation*}
\end{proof}
\begin{remark}
For $r>2$,
$\zeta(r-1)
\geq\int_1^\infty x^{1-r}\,dx
=\frac1{r-2}. $
Combining the main theorem and Proposition~\ref{prop:lower} shows that there is no universal constant $C>0$, independent of $r,p,q$, such that for any $T\in \mathcal{S}_r(\mathcal{H})$ $(r>2)$, one can choose
$A\in \mathcal{S}_p(\mathcal{H})$ and $B\in \mathcal{S}_q(\mathcal{H})$ satisfying
\begin{equation*}
T=[A,B],
\qquad
\|A\|_p\|B\|_q\leq C\|T\|_r.
\end{equation*}
\end{remark}

\section*{AI statement}
The authors used GPT 6 astra during the development of this work for searching necessary references, examining mathematical arguments, and language editing.  The authors have checked and substantially revised all mathematical arguments and take full responsibility for the final version.

\end{document}